\documentclass[12pt,leqno,a4paper]{amsart}
\usepackage{amssymb,enumerate}
\newcommand{\FF}{{\mathbb{F}}}
\newcommand{\QQ}{{\mathbb{Q}}}

\newcommand{\ba}{{\mathbf{a}}}
\newcommand{\bb}{{\mathbf{b}}}
\newcommand{\bC}{{\mathbf{C}}}
\newcommand{\bG}{{\mathbf{G}}}
\newcommand{\bH}{{\mathbf{H}}}
\newcommand{\bL}{{\mathbf{L}}}
\newcommand{\bM}{{\mathbf{M}}}

\newcommand{\bP}{{\mathbf{P}}}
\newcommand{\bT}{{\mathbf{T}}}
\newcommand{\bX}{{\mathbf{X}}}

\newcommand{\cE}{{\mathcal{E}}}

\newcommand{\ad}{{\operatorname{ad}}}
\newcommand{\df}{{\operatorname{def}}}
\newcommand{\Ind}{{\operatorname{Ind}}}
\newcommand{\Irr}{{\operatorname{Irr}}}
\newcommand{\reg}{{\operatorname{reg}}}

\newcommand{\SC}{{\operatorname{sc}}}
\newcommand{\Syl}{{\operatorname{Syl}}}

\newcommand{\Chevie}{{\sf{Chevie}}}

\newcommand\RLG{{R_\bL^\bG}}
\newcommand\RMG{{R_\bM^\bG}}
\newcommand\RTG{{R_\bT^\bG}}
\newcommand\sR{{{}^*\!R}}
\newcommand\sRMG{{\sR_\bM^\bG}}
\newcommand\sRTG{{\sR_\bT^\bG}}
\newcommand\bGa{\bG_{\mathbf{a}}}
\newcommand\bGb{\bG_{\mathbf{b}}}
\newcommand{\tbG}{{\tilde\bG}}
\newcommand{\tbL}{{\tilde\bL}}

\newcommand{\bbG}{{\overline\bG}}
\newcommand{\bbL}{{\overline\bL}}
\newcommand{\bJ}{{\bar J}}

\newcommand\blangle{{\big\langle}}
\newcommand\brangle{{\big\rangle}}
\newcommand{\Ph}[1]{\Phi_{#1}}

\newcommand{\tw}[1]{{}^{#1}\!}
\newcommand{\wst}{{\widehat{st}}}
\newcommand{\tchi}{{\tilde\chi}}

\let\al=\alpha
\let\be=\beta
\let\eps=\epsilon
\let\vhi=\varphi
\let\la=\lambda
\let\ti=\times

\newtheorem{thm}{Theorem}[section]
\newtheorem{lem}[thm]{Lemma}

\newtheorem{cor}[thm]{Corollary}
\newtheorem{prop}[thm]{Proposition}

\newtheorem{thmA}{Theorem}
\newtheorem{corA}[thmA]{Corollary}

\theoremstyle{definition}
\newtheorem{exmp}[thm]{Example}
\newtheorem{defn}[thm]{Definition}

\theoremstyle{remark}
\newtheorem{rem}[thm]{Remark}

\begin{document}

%%%%%%%%%%%%%%%%%%%%%%%%%%%%%%%%%%%%%%%%%%%%%%%%%%%%%%%%%%%%%%%%%%%%%%%%%
%%\title{Characters in quasi-isolated blocks}
\title[Jordan correspondence and block distribution]{Jordan correspondence and\\ block distribution of characters}

%%%%%%%%%%%%%%%%%%%%%%%%%%%%%%%%%%%%%%%%%%%%%%%%%%%%%%%%%%%%%%%%%%%%%%%%%

\date{\today}

\author{Radha Kessar}
\address{Department of Mathematics, University of Manchester,  Alan Turing Building, Oxford Road, Manchester  M139PL,
         United Kingdom}
\email{radha.kessar@manchester.ac.uk}
\author{Gunter Malle}
\address{FB Mathematik, TU Kaiserslautern, Postfach 3049,
         67653 Kaisers\-lautern, Germany.}
\email{malle@mathematik.uni-kl.de}

\thanks{The work of the first author was partially supported by EPSRC grant
 EP/T004592/1. The second author gratefully acknowledges financial support by
 the DFG --Project-ID 286237555--TRR 195.
 The authors would like to thank the Isaac Newton Institute for Mathematical
 Sciences for support and hospitality during the programme \emph{Groups,
 Representations and Applications} when part of the work on this paper was
 undertaken. This work was supported by: EPSRC grant number EP/R014604/1}



\dedicatory{Dedicated to Michel Enguehard}

\begin{abstract}
We complete the determination of the $\ell$-block distribution of characters
for quasi-simple exceptional groups of Lie type up to some minor ambiguities
relating to non-uniqueness of Jordan decomposition. For this, we first
determine the $\ell$-block distribution for finite reductive groups whose
ambient algebraic group defined in characteristic different from $\ell$ has
connected centre. As a consequence we derive a compatibility between
$\ell$-blocks, $e$-Harish-Chandra series and Jordan decomposition.
Further we apply our results to complete the proof of Robinson's conjecture on
defects of characters.
\end{abstract}

\maketitle

%\pagestyle{myheadings}
%%\markboth{}{}
%\markboth{for personal use only}{preliminary}

%%%%%%%%%%%%%%%%%%%%%%%%%%%%%%%%%%%%%%%%%%%%%%%%%%%%%%%%%%%%%%%%%%%%%%%%%
\section{Introduction}   \label{sec:intro}

A fundamental ingredient in the understanding of the modular representation
theory of a finite (simple) group is the distribution of its irreducible
complex characters into Brauer $\ell$-blocks for the primes $\ell$ dividing
the group order. For example, such information has recently been used in a
crucial way in the proof of several deep conjectures, like the Alperin--McKay
conjecture for the prime $\ell=2$, or Brauer's height zero conjecture.

In this paper we contribute to the determination of these $\ell$-blocks for
quasi-simple groups, in particular to the case of finite exceptional groups of
Lie type for bad primes~$\ell$. A parametrisation of these blocks (for simply
connected types) in terms of $e$-cuspidal pairs had been completed in our
previous paper \cite{KM}; here we describe the subdivision of the corresponding
Lusztig series along those blocks, first for groups arising from algebraic
groups with connected centre (Theorem~\ref{thm:thmA}), and then using Clifford
theory for the quasi-simple groups themselves (Proposition~\ref{prop:sc}) 

\par
For unipotent blocks this subdivision had already previously been obtained by
Brou\'e--Malle--Michel \cite{BMM} for large primes $\ell$, by Cabanes--Enguehard
\cite{CE94} for good primes and by Enguehard \cite{En00} in general (up to
some indeterminacies), while for arbitrary blocks at good primes $\ell$ this
question had been studied by Cabanes--Enguehard \cite{CE99}, and Enguehard
\cite{En08,En13}, albeit in a different formulation. Here, we settle
the remaining cases, thereby arriving at the following result:

\begin{thmA}   \label{thm:thmA}
 Let $\bX$ be a connected reductive group in characteristic~$p$ with connected
 centre and simple, simply connected derived subgroup, with a Frobenius map
 $F:\bX\rightarrow\bX$. Let $\bG$ be an $F$-stable Levi subgroup of $\bX$, let
 $\ell\ne p$ be a prime and $s$ be a semisimple $\ell'$-element in the dual
 group~$\bG^{*F}$. Then for every $\ell$-element
 $t\in C_{\bG^*}(s)^F$ there exists a map $\bJ_t^\bG$ from the set of unipotent
 $\ell$-blocks of $C_{\bG^*}(st)^F$ to the set of $t$-twin $\ell$-blocks in
 $\cE_\ell(\bG^F,s)$ such that if $\eta\in\cE(C_{\bG^*}(st)^F,1)\cap\Irr(b)$
 for a unipotent $\ell$-block $b$ of $C_{\bG^*}(st)^F$, then its Jordan
 correspondent in $\cE(\bG^F,st)$ belongs to $\bJ_t^\bG(b)$.
\end{thmA}
%{\bf[G: Ok, $t$ shows up in the definition, but in the end, the pair of blocks in $\bG^F$ does not depend?]}

This shows that \cite[Thm(iii)]{CE94} and \cite[Thm~B]{En00} continue to hold
for bad primes $\ell$ and non-unipotent blocks, up to the small ambiguities
around twin blocks, only occurring in type $E_8$.
Let us point out that those also arise in unipotent blocks for which it seems
they have not been resolved in \cite{En00}, either. Our map $\bJ_t^\bG$ is
similar to the one introduced in \cite{En00}; it will be explained in
Proposition~\ref{prop:J_t} and Section~\ref{subsec:t-twins}. The meaning of
``$t$-twin block'' will be
defined in Section~\ref{subsec:t-twins}. In most cases, it is just a single
$\ell$-block, but in a few cases in $\bG^F=E_8(q)$ with $q\equiv-1\bmod\ell$ it
consists of a union of two blocks, see Table~\ref{tab:twins}.

As a consequence we obtain the following compatibility between $\ell$-blocks,
$e$-Harish-Chandra series and Jordan decomposition; here $e$ is the order of
$q$ modulo~$\ell$, respectively modulo~4 when $\ell=2$:

\begin{corA}   \label{cor:cor1}
 Let $\bG$ and $\ell$ be as in Theorem~$\ref{thm:thmA}$. If
 $\chi,\chi'\in\Irr(\bG^F)$ lie in the same Lusztig series say $\cE(\bG^F,r)$,
 for some semisimple $r\in \bG^{*F}$, and the Jordan correspondents of $\chi$
 and $\chi'$ in $\cE(C_{\bG^*}(r)^F,1)$ lie in the same unipotent
 $e$-Harish-Chandra series, then $\chi$ and $\chi'$ lie in the same
 $r_\ell$-twin $\ell$-block of $\bG^F$.
 \end{corA}

On the way we also complete, correct and extend our results in \cite{KM}:
we deal with the isolated 5-blocks of $E_8(q)$ in Lusztig series indexed by
isolated 6-elements that had been omitted there, showing that all results
from \cite{KM} carry over (Theorem~\ref{thm:E8,l=5}), and we correct
information on
$3$-blocks of $E_7(q)$ and $E_8(q)$ that arose from a misinterpretation of
results in \cite{En00} (Proposition~\ref{prop:corr non-central}). In addition,
in Proposition~\ref{prop:E6E7-defgrp} we parametrise the isolated blocks in
groups of adjoint types $E_6$ and~$E_7$, and in Lemma~\ref{lem:dec RLG} we
settle the last open instances from \cite{BMM} of the decomposition of Lusztig
induction of unipotent characters.
\smallskip

\emph{Robinson's conjecture} \cite{Ro96} asserts that for any $\ell$-block $B$
of a finite group $G$ with defect group $D$ we have
$$\ell^{\df(\chi)}\ge |Z(D)|\qquad\text{for all $\chi\in\Irr(B)$}$$
with equality only when $D$ is abelian, where
$\df(\chi):=\log_\ell(|G|_\ell/\chi(1)_\ell)$ is the \emph{defect of $\chi$}
and $Z(D)$ denotes the centre of the defect group $D$. This can be considered
as a block-wise analogue of the well-known fact that $\chi(1)$ divides
$|G:Z(G)|$ for any irreducible character $\chi\in\Irr(G)$.
We combine our results on block distribution with information on defect groups
from \cite{Ruh22} to verify Robinson's conjecture for isolated 2-blocks of
exceptional type groups and thus complete the proof of this conjecture:

\begin{thmA}   \label{thm:Robinson}
 Robinson's conjecture holds for all blocks of all finite groups.
\end{thmA}

The paper is built up as follows. In Section~\ref{sec:general} we recall and
collect some background results in particular on Jordan decomposition of
characters. In Section~\ref{sec:extensions} we extend some of our earlier
results to the present setting, based upon which, in Section~\ref{sec:blocks}
we prove our main results Theorem~\ref{thm:thmA} and Corollary~\ref{cor:cor1}.
In Section~\ref{sec:sc type} we apply Clifford
theory to describe the blocks of quasi-simple groups of types $E_6$ and $E_7$.
In Section~\ref{sec:exc type} we correct and extend results in
\cite{KM}; in particular we parametrise the isolated $\ell$-blocks of simple
groups of adjoint exceptional type and prove the theorems stated in
Section~\ref{sec:extensions}. Finally, the proof of Robinson's conjecture is
given in Section~\ref{sec:Rob}.
\medskip

\noindent
{\bf Acknowledgements:} Our general approach and several of our proofs are
heavily inspired by the pioneering impressive work of Michel Enguehard, to whom
this paper is dedicated. We thank the anonymous referee for their thorough
reading of our paper and the helpful comments and questions.

%%%%%%%%%%%%%%%%%%%%%%%%%%%%%%%%%%%%%%%%%%%%%%%%%%%%%%%%%%%%%%%%%%%%%%%%%
\section{On Jordan decomposition and block distribution}   \label{sec:general}

We refer to \cite{GM20} for basic notions from Deligne--Lusztig character
theory.

%%%%%%%%%%%%%%%%%%%%%%%%%%%%%%%%%%%%%
\subsection{Notation and background results}
Throughout this subsection, $\bG$ is a connected reductive linear algebraic
group over an algebraic closure of a finite field, and
$F:\bG\rightarrow\bG$ is a Frobenius endomorphism endowing $\bG$ with an
$\FF_q$-structure for some prime power $q$. By $\bG^*$ we denote a group
in duality with $\bG$ with respect to some fixed $F$-stable maximal torus
of $\bG$, with corresponding Frobenius endomorphism also denoted by $F$.
The notions recalled here make sense independently of whether $\bG$ has
connected centre or not.
\par
For $e$ a positive integer and any $F$-stable torus $\bT\le\bG$, let $\bT_e$
denote its Sylow $e$-torus (see e.g.~\cite{BMM} for terminology on Sylow
$e$-theory). An $F$-stable Levi subgroup $\bL\le\bG$ is called \emph{$e$-split}
if $\bL=C_\bG(Z^\circ(\bL)_e)$ or equivalently if $\bL= C_\bG(\bT)$ for some
$e$-torus $\bT$ of $\bG$. A character $\la\in\Irr(\bL^F)$ is called
\emph{$e$-cuspidal} if $^*\!R_{\bM\le\bP}^\bL(\la)=0$ for all proper $e$-split
Levi subgroups $\bM<\bL$ and any parabolic subgroup $\bP$ of $\bL$ containing
$\bM$ as Levi complement, where $\tw*R_{\bM\le\bP}^\bL$ denotes Lusztig
restriction. It is known that this property is independent of the chosen
parabolic subgroup $\bP$ unless possibly if $\bG^F$ has a component of type
$\tw2E_6(2)$ or $E_8(2)$ (see \cite[Thm~3.3.8]{GM20}).
\par
Let $s\in{\bG^*}^F$ be semisimple. Choose some Jordan decomposition for
$\bG^F$ (as in \cite[Thm~11.5.1]{DM20}). Then $\chi\in\cE(\bG^F,s)$ is
\emph{$e$-Jordan-cuspidal} if $Z^\circ(C_{\bG^*}^\circ(s))_e=Z^\circ(\bG^*)_e$
and $\chi$ corresponds under Jordan decomposition to the $C_{\bG^*}(s)^F$-orbit
of an $e$-cuspidal unipotent character of $C_{\bG^*}^\circ(s)^F$. 
If $\bL\le\bG$ is $e$-split and $\la\in\Irr(\bL^F)$ is $e$-Jordan-cuspidal,
then $(\bL,\la)$ is called an \emph{$e$-Jordan-cuspidal pair} of $\bG$.
By \cite[Prop.~1.10(ii)]{CE99}, $e$-cuspidality implies
$e$-Jordan-cuspi\-dality; a list of situations where the converse is true is
given in Remark~2.2 and Section~4 of \cite{KM15} (see also
Theorem~\ref{thm:[KM13, 1.2]}(f)).

%%%%%%%%%%%%%%%%%%%%%%%%%%%%%%%%%%%%%
\subsection{Generalities on Jordan decomposition}
Suppose now that $\bG$ is connected reductive with connected centre. Let
$(\bG^*,F)$ be dual to $(\bG,F)$ with respect to a fixed duality. Fix a
semisimple element $s\in\bG^{*F}$, and denote as usual by
$\cE(C_{\bG^*}(s)^F,1)$ the set of unipotent characters of $C_{\bG^*} (s)^F$
and deviating temporarily from the standard notation, denote by
$\cE(\bG^F,\bG^{*F},s)$ (usually denoted $\cE(\bG^F,s))$ the Lusztig series of
$\Irr(\bG^F)$ corresponding to $s$. Denote by
$$\Psi_{\bG,\bG^*,s}:\cE(C_{\bG^*}(s)^F,1)\to\cE (\bG^F, \bG^{*F},s)$$
the inverse of the Jordan decomposition map of \cite[Thm~7.1]{DM90} (see also
\cite[Thm~2.1]{SFTV}). 
 
For each pair $\bL\leq\bG$, $\bL^*\leq\bG^*$ of $F$-stable Levi subgroups in
dual conjugacy classes, we fix a duality between $(\bL,F)$ and $(\bL^*,F)$
induced by the duality between $(\bG,F)$ and $(\bG^*,F)$ as described for
example in \cite[Prop.~11.4.1]{DM20} and the ensuing discussion). Note that if
$\bL_i$ (respectively $\bL_i^*$), $i=1,2$, are $\bG^F$-conjugate (respectively
$\bG^{*F}$-conjugate) $F$-stable Levi subgroups of~$\bG$ (respectively $\bG^*$)
in dual conjugacy classes, then for any $g^*\in\bG^{*F} $ such that
$\bL_1^* =\tw{g^*}\bL_2^*$, there exists $g\in\bG^F$ such that
$\bL_2=\tw{g}\bL_1$ and such that conjugation by $g$ and $g^*$ are dual
isomorphisms (in the sense of \cite[Sec.~1.7.11]{GM20} or
\cite[Def.~4.4]{SFTV}) from $\bL_1$ to $\bL_2$ and from $\bL_2^*$ to $\bL_1^*$,
respectively. In this situation, if $s\in\bL_1^{*F}$, then setting
$t =g{^*}^{-1}sg^*$, \cite[Thm~7.1(vi)]{DM90} gives 
$$\Psi_{\bL_2,\bL_2^*,t}(\chi) = \tw{g}(\Psi_{\bL_1,\bL_1^*,s} (\tw{g^*}\chi))
  \qquad\text{for all $\chi\in\cE(C_{\bL_2^*}(t)^F, 1)$.}\leqno{(\dagger)}$$
Here $\tw{g^*}\chi$ is the character of $C_{\bL_1^*}(s)^F$ defined by
$\tw{g^*}\chi(x) = \chi({g^*}^{-1}xg^*)$, $x\in C_{\bL_1^*}(s)^F$, and
similarly, for any $\tau\in\Irr(\bL_1^F)$, $\tw{g}\tau$ is the character
of $\bL_2^F$ defined by $\tw{g}\tau(y) = \tau(g^{-1}yg)$, $y\in\bL_2^F$.
Further, if $\bL= \bL_1=\bL_2$, $\bL^*= \bL_1^*=\bL_2^*$, and $g^*\in\bL^{*F}$
then $g$ can be chosen in $\bL^F$. In this case, since inner
automorphisms act trivially on characters, the above equation yields
$\Psi_{\bL,\bL^*,t}(\chi) = \Psi_{\bL,\bL^*,s}(\tw{g^*}\chi)$ for all
$\chi\in\cE_\ell(C_{\bL^*}(t)^F,1)$.

We say that a pair $(\bL,\la)$, where $\bL\leq\bG$ is an $F$-stable Levi
subgroup and $\la\in\Irr(\bL^F)$, \emph{lies below} $(\bG^F,s)$ if
$\la\in\cE(\bL^F,\bL^{*F},s)$ for some $\bL^*$ containing $s$ and in duality
with $\bL$.
We say that a $\bG^F$-class of pairs $(\bL,\la)$ lies below $(\bG^F,s)$ if
some element of the class lies below $(\bG^F,s)$. Let $\bC^*:=C_{\bG^*}(s)$.

\begin{lem}   \label{lem:genjor}
 There is a bijection between $\bC^{*F}$-classes of pairs $(\bL^*,\mu)$, for
 $\bL^*\le\bG^*$ an $F$-stable Levi subgroup with $s\in\bL^*$ and
 $\mu\in\cE(C_{\bL^*}(s)^F,1)$, and $\bG^F$-classes of pairs below $(\bG^F,s)$.
 It sends the $ \bC^{*F}$-class of $(\bL^*,\mu)$ to the $\bG^F$-class of
 $(\bL,\Psi_{\bL,\bL^*,s}(\mu))$ where $\bL$ is dual to $\bL^*$.
\end{lem}

\begin{proof}
The proof is a consequence of ($\dagger$). We give the details. First note that
for any dual pair $\bL, \bL^*$ of $F$-stable Levi subgroups with $s \in \bL^*$
and any $\mu\in\cE(C_{\bL^*}(s)^F,1)$, the pair $(\bL,\Psi_{\bL,\bL^*,s}(\mu))$
lies below $(\bG^F,s)$.
Suppose that $(\bL_i,\la_i)$, $i=1,2$ are two pairs such that $\bL_i$ is dual
to $\bL^*$ and $\la_i = \Psi_{\bL_i,\bL^*,s}(\mu)$. Applying ($\dagger$) with
$\bL_2^*= \bL_1^*$ and $g^*=1$ yields that $(\bL_1,\la_1)$ and $(\bL_2,\la_2)$
are $\bG^F$-conjugate. Thus there is a well-defined and surjective map from the
set of pairs $ (\bL^*,\mu)$, $\bL^*$ an $F$-stable Levi subgroup of $\bG^*$ with
$s\in \bL^*$ and $\mu\in\cE(C_{\bL^*}(s)^F,1)$, to the set of $\bG^F$-classes
of pairs below $(\bG^F,s)$ sending the pair $(\bL^*,\mu)$ to the class of
$(\bL,\Psi_{\bL,\bL^*,s}(\mu))$ where $\bL$ is dual to $\bL^*$.
\par
Now let $(\bL_i^*,\mu_i)$, $i=1,2$, be such that $\bL_i^*$ is an $F$-stable Levi
subgroup of $\bG^*$, $s\in\bL_i^*$, $\mu_i\in\cE(C_{\bL_i^*}(s)^F,1)$.
If $(\bL_1^*,\mu_1)$ and $ (\bL_2^*, \mu_2)$ are $\bC^{*F}$-conjugate, say
$(\bL_2^*,\mu_2)=\,^x(\bL_1^*,\mu_1)$, $x\in \bC^{*F}$, then applying
($\dagger$) with $g^*= x^{-1}$ and any $\bL_1 =\bL_2$ dual to the $\bL_i^*$
yields that $(\bL_1^*,\mu_1)$ and $ (\bL_2^*,\mu_2)$ have the same image under
the above map. Conversely, suppose that $(\bL_1^*,\mu_1)$ and $(\bL_2^*,\mu_2)$
have the same image. To complete the proof it suffices to show they are
$\bC^{*F}$-conjugate. Let $\bL$ be dual to the $\bL_i^*$s and set
$\la_i = \Psi_{\bL, \bL_i^*,s}(\mu_i)$. By hypothesis, there exists
$g\in N_\bG(\bL)^F$ such that
$\la_2=\tw{g}\la_1 =\tw{g} \Psi_{\bL,\bL_1^*,s}(\mu_1)$. So, by~($\dagger$),
there exists $g^*\in\bG^{*F}$ such that $\bL_1^* = \tw{g^*}\bL_2^*$ and 
$$\la_2= \tw{g}(\Psi_{\bL,\bL_1^*,s}(\mu_1))
  = \Psi_{\bL,\bL_2^*,t}(\tw{{g^*}^{-1}}\mu_1)$$
with $t=g^{*^{-1}}s g^*$. In particular, $\la_2\in\cE(\bL^F\!,\bL_2^{*F}\!,t)$.
Since by assumption $\la_2\in\cE(\bL^F\!,\bL_2^{*F}\!\!,s)$, $t$ and $s$ are
$\bL_2^{*F}$-conjugate, say $t = {h^*}^{-1}s h^*$ with $h^*\in\bL_2^{*F}$.
Thus, by the remarks after~($\dagger$) applied with $h^*$ in place of $g^*$,
$\bL_2^*$ in place of $\bL^*$ and $\tw{{g^*}^{-1}}\mu_1$ in place of $\chi$,
$$\la_2=\Psi_{\bL,\bL_2^*,t}(\tw{{g^*}^{-1}}\mu_1)
  =\Psi_{\bL,\bL_2^*,s} (\tw{h^*{g^*}^{-1}}\mu_1).$$
Since we also have $\la_2 =\Psi_{\bL,\bL_2^*,s}(\mu_2)$, it follows that
$\mu_2 = \tw{h^*{g^*}^{-1}}\mu_1$. As $h^*\in\bL_2^{*F}$ we obtain
$$(\bL_2^*,\mu_2) = \tw{h^*{g^*}^{-1}}(\bL_1^*,\mu_1) $$
with $h^*{g^*}^{-1}\in\bC^{*F}$, as required.
\end{proof}

From now on for any pair of dual $F$-stable Levi subgroups $\bL,\bL^*$ with
$s\in \bL^*$, we revert to the notation $\cE(\bL^F,s)$ for
$\cE(\bL^F,\bL^{*F},s)$ and if $s\in \bL^{*F}$ we denote by
$$\pi_{s}^\bL:\cE(\bL^F,s)\longrightarrow\cE(C_{\bL^*}(s)^F,1)$$
the Jordan decomposition inverse to $\Psi_{\bL,\bL^*,s}$.

We are going to define some maps between $e$-Jordan cuspidal pairs which will
then induce maps between $\ell$-blocks.

\begin{prop} \label{prop:En 15}
 For any $e\ge1$, there is a natural bijection between the $\bC^{*F}$-classes
 of unipotent $e$-cuspidal pairs in $\bC^*$ and the $\bG^F$-classes of
 $e$-Jordan cuspidal pairs below $(\bG^F,s)$. It sends the class of
 $(\bL_s^*,\la_s)$ to the class of $(\bL,\la)$, with
 $\bL^*=C_{\bG^*}(Z^\circ(\bL_s^*)_e)$ and $\la_s=\pi_s^\bL(\la)$, where
 $\bL,\bL^*$ are in duality.
\end{prop}

\begin{proof}
First of all, we note that the maps $\bM^*\mapsto C_{\bG^*}(Z^\circ(\bM^*)_e)$,
$\bL^*\mapsto C_{\bL^*}(s)$ are mutually inverse bijections between the set of
$e$-split Levi subgroups $\bM^*$ of $\bC^* $ and the set of $e$-split Levi
subgroups $\bL^*$ of $\bG^* $ which contain $s$ and for which additionally 
$ Z^\circ(C_{\bL^*}(s))_e= Z^\circ(\bL^*)_e $. These bijections are compatible
with the action of $\bC^{*F}$ and hence induce inverse bijections between the
set of $\bC^{*F}$-classes of $e$-split Levi subgroups of $\bC^* $ and the set
of $\bC^{*F}$-classes of $e$-split Levi subgroups $\bL^*$ of $\bG^*$
which contain $s$ and for which additionally
$Z^\circ(C_{\bL^*}(s))_e= Z^\circ(\bL^*)_e$.
\par
Now the composition of the bijection which sends the $\bC^{*F}$-class of
$(\bL_s,\la_s)$ to the $\bC^{*F}$-class of $(C_{\bG^*}(Z^\circ(\bL_s)_e),\la_s)$
with the bijection of Lemma~\ref{lem:genjor} yields the result.
\end{proof}

%%%%%%%%%%%%%%%%%%%%%%%%%%%%%%%%%%%%%
\subsection{Compatibility with central products}   \label{subsec:compatibility}
We state a compatibility result of central products with various constructions
which surely is well-known to the experts. We begin with a general fact. Recall
that an isotypy between algebraic groups $\bX_0$, $\bX$ is a morphism of
algebraic groups $f:\bX_0 \to\bX$ with central kernel and with image containing
$[\bX,\bX]$.

\begin{lem}   \label{lem:centraliser-isotypy}
 Let $f:\bX_0\to\bX$ be an isotypy of connected reductive algebraic groups. 
 \begin{enumerate} [\rm(a)]
  \item The map $\bL\mapsto\bL_0:= f^{-1}(\bL)$ induces a bijection between the
   sets of Levi subgroups of~$\bX$ and of~$\bX_0$. If
   $\bL$ and $\bL_0$ correspond, then $\bL= Z(\bX)f(\bL_0)$ and
   $[\bL,\bL]= f([\bL_0,\bL_0])$. In particular, $f:\bL_0\to\bL$ is an isotypy.
  \item Let $s\in\bX_0$ be a semisimple element such that $C_{\bX_0}(s)$ and
   $C_\bX(f(s))$ are both connected. Then the induced map
   $C_{\bX_0}(s)\to C_\bX(f(s))$ is an isotypy. Suppose that $\bL_0$ is a
   Levi subgroup of $\bX_0$ containing $s$ and let $\bL\leq\bX$ be the Levi
   subgroup corresponding to $\bL_0$ via $f$. Then the Levi subgroups
   $C_{\bL_0}(s) \leq C_{\bX_0}(s)$ and  $C_\bL(f(s)) \leq C_\bX(f(s))$ 
   correspond via the induced isotypy between $C_{\bX_0}(s)$ and $C_\bX(f(s))$.
 \end{enumerate}
\end{lem} 

\begin{proof}
Part~(a) is well-known. For the first assertion of (b), it suffices to show
that $f(C_{\bX_0}(s))$ contains $[C_\bX(f(s)), C_\bX(f(s))]$. Let
$\bX_1:=f^{-1}(C_\bX(f(s)))$. Since $[\bX,\bX]\leq f(X_0)$,
$C_{[\bX,\bX]}(f(s))\leq f(\bX_1)$. On the other hand, since
$\bX = Z(\bX)[\bX,\bX]$, we have $C_\bX(f(s)) =Z(\bX)C_{[\bX,\bX]}(f(s))$.
Thus, $C_\bX(f(s)) = Z(\bX)f(\bX_1)$.  Since $\ker(f)$ is central in $\bX_0$,
$[\bX_1,s]\leq Z(\bX_0)$ and the map $\bX_1\to Z(\bX_0)$ sending
$x\in\bX_1$ to $[x,s]$ is a group homomorphism with kernel $C_{\bX_0}(s)$.
In particular, $\bX_1/C_{\bX_0}(s)$ is abelian, hence so is
$f(\bX_1)/f(C_{\bX_0}(s)) $ and we obtain 
$$f(C_{\bX_0}(s)) \geq [f(\bX_1), f(\bX_1)] = [C_\bX(f(s)), C_\bX(f(s))]. $$
The second assertion of (b) follows from part (a).
\end{proof}

Suppose that $\bG=\bG_1\bG_2$ is a central product of connected reductive,
connected centre and $F$-stable subgroups $\bG_i$, $i=1,2$. Let
$\tbG= \bG_1\times\bG_2$ and let $\vhi:\tbG\to\bG$ be the canonical epimorphism
given by multiplication. We assume $\ker(\vhi)$ is a (central) torus of
$\tilde\bG$. In particular, $\vhi$ is an isotypy. Let $\vhi^*:\bG^*\to\tbG^*$
denote a dual isotypy and note that $\tbG^* = \bG_1^*\times\bG_2^*$, with
$\bG_i^*$ dual to $\bG_i$. 
For an $F$-stable Levi subgroup $\bL$ of $\bG$ with dual Levi subgroup
$\bL^* \leq \bG^*$, let $\tbL:=\vhi^{-1}(\bL)$ be the corresponding Levi
subgroup of $\tbG$. Then $\tbL= \bL_1 \times \bL_2$ with $\bL_i$ an $F$-stable
Levi subgroup of $\bG_i$, with dual $\tbL^*:= Z^\circ(\tbG^*)\vhi^*(\bL^*)$ of
the form $\bL_1^*\times \bL_2^*$, with $\bL_i^*$ dual to $\bL_i$. Note that the
Levi subgroups $\bL^*\leq\bG^*$ and $\tbL^*\leq\tbG^*$ correspond to each other
via $\vhi^*$.

Let $s\in\bL^{*F}$ be semisimple. Then $\vhi^*(s)=(s_1,s_2)$ with
$s_i\in\bL_i^{*F}$ and $C_{\tbL^*}(\vhi^*(s)) = C_{\bL_1^*}(s_1)\times C_{\bL_2^*}(s_2)$. Further, $C_{\bL^*} (s) $ and $C_{\tbL^*}(\vhi^*(s))$ are
corresponding Levi subgroups under the isotypy $\vhi^*:C_{\bG^*}(s)\to C_{\tbG^*}(\vhi^*(s))$ (see Lemma~\ref{lem:centraliser-isotypy}). So, 
$\vhi^*$ induces a canonical bijection (see \cite[Prop.~11.3.8]{DM20})
$$\widehat{\ }:\cE(C_{\bL^*}(s)^F,1)\to\cE(C_{\tbL^*}(\vhi^*(s))^F,1);$$
we write $\hat\al=\al_1\otimes\al_2$ with $\al_i\in\cE (C_{\bL_i^*}(s_i)^F,1)$.
For $\la\in \cE(\bL^F,s)$, set $\tilde\la := \la\circ\vhi$.
Then $\tilde\la\in\cE(\tbL^F,\vhi^*(s))$ and is of the form
$\la_1\otimes\la_2$, $\la_i\in\cE(\bL_i^F,s_i)$. By properties of Jordan
decomposition \cite[Thm~7.1(vi),(vii)] {DM90}, we have that if
$\al= \pi_s^\bL (\la)$, then $\hat\al = \pi_{\vhi^*(s)}^\tbL(\tilde\la)$ and
$\al_i =\pi_s^{\bL_i}(\la_i)$, $i=1,2$.

\begin{prop}   \label{prop:JoEn 2.1.5}
 Let $\bL \leq \bG$ be an $F$-stable Levi subgroup with dual Levi subgroup
 $\bL^* \leq \bG^* $ and let $ s \in \bL^{*F}$ be semisimple. With the notation
 above, and denoting by $X'$ the derived subgroup of a group $X$, we have the
 following.
 \begin{enumerate}[\rm(a)]
  \item $(\bL,\la)$ is an $e$-Jordan cuspidal pair for $\bG$ if and only if
   $(\tbL, \tilde \la )$ is an $e$-Jordan cuspidal pair for $\tbG$, if and only
   if $(\bL_i,\la_i)$ are $e$-Jordan cuspidal pairs for $\bG_i$, $i=1,2$. 
  \item Suppose that $(\bL,\la)$ is an $e$-Jordan cuspidal pair for $\bG$. If
   $(\bL,\la)$ corresponds to the unipotent $e$-cuspidal pair
   $(C_{\bL^*}(s),\al)$ of $C_{\bG^*}(s)$ via Proposition~$\ref{prop:En 15}$,
   then $(\tbL,\tilde\la)$ corresponds to the unipotent $e$-cuspidal pair
   $(C_{\tbL^*}(\vhi^*(s)),\hat\al)$ of $C_{\tbG^*}(\vhi^*(s))$ and
   $(\bL_i,\la_i)$ correspond to the unipotent $e$-cuspidal pairs
   $(C_{\bL_i^*}(s_i),\al_i)$ of $C_{\bG_i^*}(s_i)$ for $i=1,2 $. 
  \item Let $\bM\leq\bG$ be an $F$-stable Levi subgroup with
   $\bL\leq \bM$ and let $\chi\in\Irr(\bM^F)$. Then, keeping the notational
   convention as above, $\chi $ is a constituent of $R_\bL^\bM(\la)$ if and
   only if $\tchi=\chi_1\otimes\chi_2$ with $\chi_i$ a constituent of
   $R_{\bL_i}^{\bM_i}(\la_i)$.
  \item Let $\bM\leq\bG$ be an $F$-stable Levi subgroup with dual Levi
   subgroup  $\bM^*\leq\bG^*$ containing $s$ and let
   $\be\in\cE(C_{\bM^*}(s)^F,1)$. Suppose that
   $$(C_{\bM^*}(s)',\be|_{{C_{\bM^*}(s)'}^F})\qquad\text{and}\qquad
     (C_{\bL^*}(s)',\al|_{{C_{\bL^*}(s)'}^F})$$
   are $C_{\bG^*} (s)^F$-conjugate. Then,
   $$(C_{\tilde \bM^*}(\vhi^*(s))', \hat\be|_{{C_{\tilde \bM^*}(\vhi^*(s))'}^F })\qquad\text{and}\qquad 
 (C_{\tilde \bL^*} (\vhi^*(s))',\hat \al|_{{C_{\tilde \bL^*} (\vhi^*(s))'}^F})$$
   are $C_{\tbG^*}(\vhi^*(s))^F$-conjugate and consequently,
   $$(C_{\bM_i^*}(s_i)',\be_i|_{{C_{\bM_i^*}(s_i)'}^F})\qquad\text{and}\qquad 
 (C_{\bL_i^*} (s_i)', \al_i|_{{C_{\bL_i^*} (s_i)'}^F})$$
   are $C_{\bG_i^*}(s_i)^F$-conjugate for $i=1,2$.
  \end{enumerate}
\end{prop} 

\begin{proof}
Parts (a) and (b) are implicit in \cite[Prop.~2.1.5]{En13} and part (c) is a
consequence of the commutation of Lusztig induction with $\vhi$, see
\cite[Cor.~9.2]{DM90}. The first assertion of part (d) follows from
Lemma~\ref{lem:centraliser-isotypy} and the bijection between unipotent
characters induced by isotypies. The second assertion follows from the first
through properties of direct products.
\end{proof}

%%%%%%%%%%%%%%%%%%%%%%%%%%%%%%%%%%%%%
\subsection{Generalities on block distribution} 
Let $\ell$ be a prime not dividing $q$ and let $e:=e_\ell(q)$ denote the order
of $q$ modulo $\ell$ if $\ell$ is odd and the order of $q$ modulo~$4$ if
$\ell=2$. We denote by $\cE(\bG^F,\ell')$ the set
of irreducible characters of $\bG^F$ lying in a Lusztig series $\cE(\bG^F,s)$
for some semisimple $\ell'$-element $s\in\bG^{*F}$. Recall from
\cite[Def.~2.4]{KM} that a character $\chi\in\cE(\bG^F,\ell')$ is said to be of
\emph{central $\ell$-defect} if $|\chi(1)|_\ell |Z(\bG)^F|_\ell =|\bG^F|_\ell$
or equivalently if the $\ell$-block of $\bG^F$ containing $\chi$ has a central
defect group (see \cite[Prop.~2.5(c)]{KM}) and $\chi$ is said to be of
\emph{quasi-central $\ell$-defect} if some (and hence any)
character of $[\bG, \bG]^F$ covered by $\chi$ is of central $\ell$-defect. By
\cite[Prop.~2.5(b)]{KM}, if $\chi$ is of central $\ell$-defect, then it is of
quasi-central $\ell$-defect. For unipotent characters, the converse holds and
the terms ``quasi-central $\ell$-defect" and ``central $\ell$-defect" are used
interchangeably:

\begin{lem}   \label{lem:qcisc}
 Any $\chi\in\cE(\bG^F,1)$ of quasi-central $\ell$-defect is of central
 $\ell$-defect and is also $e$-cuspidal.
\end{lem}

\begin{proof}
Let $b$ be the $\ell$-block of $\bG^F$ containing $\chi$. By
\cite[Prop.~2.5(f)]{KM}, $\chi$ is the unique unipotent character in its
$\ell$-block. On the other hand, by \cite[Thms~A and A.bis]{En00},
$\Irr(b)\cap\cE(\bG^F,1)$ contains the $e$-Harish-Chandra series of $\bG^F$
above some $e$-cuspidal unipotent pair $(\bL,\la)$ with $\la$ of central
$\ell$-defect. But this means that $\bL=\bG $ and $\la=\chi $.
\end{proof} 

We record here another fact that will be used later.

\begin{lem}   \label{lem:2unipquasicentraldefect}
 Suppose that all components of $[\bG,\bG]$ are of classical type and let
 $\chi\in\cE(\bG^F,1)$. If $\chi$ is of (quasi)-central $2$-defect, then $\bG$
 is a torus.
\end{lem}

\begin{proof}
By \cite[Thm~13]{CE93} and the table in \cite[p.~348]{En00}, the principal
2-block is the only unipotent 2-block of ~$\bG^F$. Hence the hypothesis implies
that the principal $2$-block of $\bG^F$ has central defect groups, that is,
the Sylow $2$-subgroups of $\bG^F$ are contained in $Z(\bG^F)$.
The result follows.
\end{proof}

Now fix $s\in\bG^{*F}$ a semisimple $\ell'$-element. The following is an
extension of the inductive argument given on Page~367 of \cite{En00}. For an
$F$-stable subgroup $\bH$ of $\bG$ we denote by $d^{1,\bH}$ the
decomposition map on the set of class functions of $\bH^F$ defined by
$d^{1,\bH}(\chi)(x): = \chi(x)$ if $x\in\bH^F$ is $\ell$-regular and
$d^{1,\bH}(\chi)(x): = 0$ otherwise. Note that we are diverging from the
more customary (but also more clumsy) notation $d^{1,\bH^F}$.

\begin{lem}   \label{lem:in Levi}
 Let $t\in C_{\bG^*}(s)^F$ be an $\ell$-element and $\chi\in\cE(\bG^F,st)$.
 Then there exists $\gamma\in\cE(\bG^F,s)$ with
 $\langle d^{1,\bG}(\chi),\gamma\rangle\ne0$.
\end{lem}

\begin{proof}
The regular character of $\bG^F$ is a linear combination of class functions of
the form $\RTG(\reg_{\bT^F})$, where $\bT$ runs over the $F$-stable maximal tori
of $\bG$ and $\reg_{\bT^F}$ is the regular character of $\bT^F$ (see
\cite[Cor.~10.2.6]{DM20}). Hence, there exists some $F$-stable maximal torus
$\bT\le\bG$ such that 
$$a:= \sRTG(\chi)(1) = \blangle \sRTG(\chi),\reg_{\bT^F} \brangle
  = \blangle\chi,\RTG (\reg_{\bT^F})\brangle\ne0. $$
Let $\sRTG(\chi) = \sum_{\sigma, \tau} a_{\sigma,\tau}\sigma\tau$ where
$\sigma$ runs over $\Irr(\bT^F)_{\ell'}$ and $\tau$ runs over
$\Irr(\bT^F)_\ell$ and note that $a = \sum_{\sigma,\tau}a_{\sigma,\tau}$.
Let $\sigma\in\Irr(\bT^F)_{\ell'}$ with $\sum_\tau a_{\sigma,\tau}\ne 0$.
For any $\tau\in\Irr(\bT^F)_\ell$, we have
$d^{1,\bT}(\sigma\tau) = d^{1,\bT}(\sigma)$ and hence 
$$\blangle d^{1,\bG}(\chi),\RTG(\sigma)\brangle
  = \blangle d^{1, \bT}(\sRTG(\chi)), \sigma \brangle
  = \frac{1}{|\bT^F_\ell|} \sum_{\tau} a_{\sigma, \tau} \ne 0. $$
Thus, there is a constituent $\gamma$ of $\RTG(\sigma)$ with
$\blangle d^{1,\bG}(\chi),\gamma\brangle\ne 0$. By the above displayed
equation, we have $a_{\sigma,\tau}\ne 0$ for some $\tau\in \Irr(\bT^F)_\ell$,
implying $\gamma\in\cE(\bG^F,s)$ as desired.
\end{proof}

Now assume $\bG$ has connected centre. For any $\ell$-element $t\in\bC^{*F}$
recall the Digne--Michel Jordan decomposition
$$\pi_{st}^\bG:\cE(\bG^F,st)\longrightarrow\cE(C_{\bG^*}(st)^F,1)$$
%%  =\cE(C_{\bC^*}(t)^F,1)$$
discussed above. For any maximal torus $\bT^*\le C_{\bG^*}(s)$ denote by
$\hat s\in\Irr(\bT^F)$ the linear character associated to $s$ by duality
\cite[(8.14)]{CE}.

\begin{lem}   \label{lem:RTG}
 Let $t\in C_{\bG^*}(s)^F$ be an $\ell$-element and $\chi\in\cE(\bG^F,st)$. Let
 $\bT^*\le C_{\bG^*}(st)$ be a maximal torus such that
 $\langle\pi_{st}^\bG(\chi),R_{\bT^*}^{C_{\bG^*}(st)}(1)\rangle\ne0$ and
 $\bT^*$ is $C_{\bG^*}(st)^F$-conjugate to any of its $\bG^{*F}$-conjugates
 contained in $C_{\bG^*}(st)$.
 Then $\chi$ lies in the same $\ell$-block of $\bG^F$ as some
 constituent of $\RTG(\hat s)$.
\end{lem}

\begin{proof}
Since $\langle\pi_{st}^\bG(\chi),R_{\bT^*}^{C_{\bG^*}(st)}(1)\rangle\ne0$, by
the properties of Jordan decomposition we have
$\langle\chi,\RTG(\wst)\rangle\ne0$, for $\bT$ dual to $\bT^*$. Now under our
assumption on $\bT^*$, by \cite[Lemma~21]{En00} we have
$$\sRTG(\chi)=m\ \sum_{w}\wst^w$$
for $w$ ranging over certain elements of the relative Weyl group of $\bT$,
including $w=1$, and where $m\ne0$. Then
$$\blangle d^{1,\bG}(\chi),\RTG(\hat s)\brangle
  =\blangle d^{1,\bT}(\sRTG(\chi)),d^{1,\bT}(\hat s)\brangle
  =m\cdot \sum_{w}\blangle d^{1,\bT}(\hat s)^w,d^{1,\bT}(\hat s)\brangle\ne0.
$$
Thus $\chi$ lies in the $\ell$-block of some constituent of $\RTG(\hat s)$.
\end{proof}

The next result uses the notation $\bGa$, $\bGb$ introduced in
\cite[Not.~2.3]{CE94}. Recall that $Z^\circ(\bG)\le\bG_\ba$ while $\bG_\bb$ is
semisimple.

\begin{lem}   \label{lem:Ga}
 Suppose that $Z(\bG)$ is connected, $\bG=\bG_\ba $ and that $\ell$ is odd. Let
 $s\in\bG^{*F}$ be a semisimple $\ell'$-element. Then there is a unique
 $\bG^F$-conjugacy class of $e$-Jordan cuspidal pairs below $(\bG^F,s)$ and a
 unique $\ell$-block in $\cE_\ell(\bG^F,s)$, say $b$. Further, denoting by
 $(\bL,\la)$ an $e$-Jordan cuspidal pair below $(\bG^F,s)$, every element of
 $\Irr(b)\cap\cE(\bG^F,s)$ is a constituent of $\RLG(\la)$.
\end{lem}

\begin{proof}
We have $C_{\bG^*}(s)_\ba = C_{\bG^*}(s)$ (see for instance remark after
\cite[Not.~2.3]{CE94}), hence by \cite[Prop.~3.3] {CE94}, there is only one
class of unipotent $e$-cuspidal pairs in $C_{\bG^*}(s)$. The first assertion
now follows from Proposition~\ref{prop:En 15}. The second assertion is a
consequence of the first, the main theorem of \cite{CE99}, and the fact that
an $e$-cuspidal pair is also $e$-Jordan cuspidal (see \cite[Sec.~1.3]{CE99}).
Here we note that in \cite{CE99} it is assumed that $\ell\geq 5$, but this is
not necessary for the case that we are considering (the second assertion can
also be obtained from the main theorem of \cite{CE94} in combination with the
Bonnaf\'e--Rouquier Jordan decomposition theorem since $C_{\bG^*}(s)$ is a Levi
subgroup of $\bG^*$ in our case). The final assertion follows from the main
theorem of \cite{CE94} applied to $C_{\bG^*}(s)$ and the fact that Jordan
decomposition commutes with Lusztig induction in groups of type $A$ (see
\cite[Thm~4.7.2]{GM20}).
\end{proof} 

For later use, we record here the following structural result.

\begin{prop}   \label{prop:3D4}
 Let $\bH$ be connected reductive with $[\bH,\bH]$ simple, $F:\bH\to\bH$ a
 Frobenius map, $\bG\le\bH$ an $F$-stable Levi subgroup and let $s\in\bG^*$ be
 semisimple. If $C_{\bG^*}(s)$ has an $e$-split Levi subgroup, $e\in\{1,2\}$,
 whose $F$-fixed points have a component of type $\tw3D_4$, then one of the
 following hold:
 \begin{enumerate}[\rm(1)]
  \item $[\bH,\bH]=[\bG,\bG]$ is of type $D_4$, $F$ induces triality and
   $s$ is central; or
  \item $\bH$ is of exceptional type, and either $\bG$ is also of exceptional
   type $E_n$, $6\le n\le 8$, or $\bG^F$ has a component of type $\tw3D_4$ and
   $s$ is central in $\bG^*$ or non-isolated.
 \end{enumerate}
\end{prop}

\begin{proof}
Clearly we may replace $\bH$ by $[\bH,\bH]$ and $\bG$ by $\bG\cap[\bH,\bH]$.
Now first assume $\bH$ is of classical type. Let $\bT_0$ be a maximally split
torus of $\bH$, with Weyl group $W$, set of simple reflections $S$, and $F$
acting by $\sigma$ on $W$. Then there is some subset
$I\subseteq S$ and $w\in N_{W}(W_I)$ such that $\bG$ has Weyl group $W_I$ and
$F$ acts by $w\sigma$ on $W_I$. Now by the explicit description in
\cite[p.~71]{How} the normalisers of parabolic subgroups of $W$ of type $D_n$
do not induce a triality automorphism on any simple factor, so nor does
$w\sigma$ unless $\sigma$ itself is triality. Thus, except we are in the
excluded case, $\bG$ is a product of groups of classical type such that $\bG^F$
has no component of type $\tw3D_4$. But the centralisers of semisimple elements
in finite classical groups only possess classical components
(see e.g.~\cite[Sect.~1]{FS89}), and their $e$-split Levi subgroups then have
the same property.   \par
Now assume $\bH$ is of exceptional type and $C_{\bG^*}(s)$ has an
$e$-split Levi subgroup whose $F$-fixed points have a component of type
$\tw3D_4$. If $\bG$ itself is not of exceptional type, then by rank
considerations it has at most one factor of type $D_n$, where $4\le n\le8$, and
type $A$-factors otherwise. In fact, there must be exactly one type $D_n$-factor
since Levi subgroups of element centralisers in type $A$-groups do not have
$\tw3D_4$-components. The $F$-fixed points of the unique $D_n$-factor are
either finite orthogonal groups, but as seen above these do not possess
semisimple elements with suitable Levi subgroups, or we have $n=4$ and $F$
induces triality, so $\bG^F$ has a component of type $\tw3D_4$. As all other
components of $\bG$ are of type $A$, their only isolated elements are central,
whence our claim.
\end{proof}

\begin{lem}   \label{lem:3D4b}
 Let $\bG$ be simple of type $D_4$ in characteristic different from~$3$ and
 $F:\bG\to\bG$ with $\bG^F=\tw3D_4(q)$. Let $\bT\le\bG$ be the centraliser of
 a Sylow $e_3(q)$-torus. Then for all $3$-elements $1\ne t\in\bG^{*F}$,
 $(\bT,1)$ is the unique unipotent $e_3(q)$-cuspidal pair of $C_{\bG^*}(t)$ (up
 to conjugation).
\end{lem}

\begin{proof}
By inspection of \cite[Tab.~2.2]{DM87}, for all $3$-elements $t\ne 1$,
$C_{\bG^*}(t)=C_{\bG^*}(t)_\ba$ contains a conjugate of $\bT$. The
claim follows.
\end{proof}

%%%%%%%%%%%%%%%%%%%%%%%%%%%%%%%%%%%%%%%%%%%%%%%%%%%%%%%%%%%%%%%%%%%%%%%%%
\section{On the $\ell$-block distribution of $\ell'$-characters}   \label{sec:extensions}

In this section we formulate some straightforward extensions of results
obtained in our predecessor papers \cite{KM} and \cite{KM15} to the following
setting: let $\bX$ be connected reductive with connected centre such that
$[\bX,\bX]$ is simple of simply connected type, with a Frobenius map
$F:\bX\to\bX$ with respect to an $\FF_q$-structure. Let $\bG\le\bX$ be an
$F$-stable Levi subgroup. Thus, in particular, all centralisers of semisimple
elements in $\bG^*$ are connected.
Let $\ell{\not|}q$ be a prime and set $e=e_\ell(q)$.
\par
Recall that a character $\chi\in\cE(\bG^F,\ell')$ is
called \emph{$e$-Jordan quasi-central cuspidal} if $\chi$ is $e$-Jordan
cuspidal and its Jordan correspondent is of quasi-central $\ell$-defect (see
\cite[Def.~2.12]{KM15}). An \emph{$e$-Jordan quasi-central cuspidal pair of
$\bG$} is a pair $(\bL,\la)$ such that $\bL$ is an $e$-split Levi subgroup of
$\bG$ and $\la\in\cE(\bL^F,\ell')$ is an $e$-Jordan quasi-central cuspidal
character of $\bL^F$. 

%%%%%%%%%%%%%%%%%%%%%%%%%%%%%%%%%%%%%
\subsection{Parametrisation of $\ell$-blocks in connected centre groups}
Parts~(a), (b) and~(c) of the following extend Theorems~1.2(a) and~1.4 of
\cite{KM} whilst parts~(d) and~(e) were implicit from the computations made in
\cite{KM} (where $\bX$ was assumed simple of simply connected exceptional type):

\begin{thm}   \label{thm:[KM13, 1.2]}
 Assume $\bG=\bX$ above and that $\ell$ is bad for $\bG$, and let
 $s\in\bG^{*F}$ be an isolated semisimple $\ell'$-element. Then:
 \begin{enumerate}[\rm(a)]
  \item There is a natural bijection
   $b_{\bG^F}(\bL,\la)\longleftrightarrow(\bL,\la)$ between $\ell$-blocks of
   $\bG^F$ in $\cE_\ell(\bG^F,s)$ and $e$-cuspidal pairs $(\bL,\la)$ below
   $(\bG^F,s)$ of quasi-central $\ell$-defect.
  \item The sets $\cE(\bG^F,(\bL,\la))$, where $(\bL,\la)$ runs over a set of
   representatives of the $\bG^F$-classes of $e$-cuspidal pairs below
   $(\bG^F,s)$, partition $\cE(\bG^F,s)$.
  \item $\bG^F$ satisfies an $e$-Harish-Chandra theory above each $e$-cuspidal
   pair $(\bL,\la)$ below $(\bG^F\!,s)$.\!\!\!
  \item If two $e$-cuspidal pairs below $(\bG^F,s)$ define the same block
   of $\bG^F$, then their Jordan correspondents define the same unipotent block
   of $C_{\bG^*}(s)^F$, up to twins.
  \item For any $e$-cuspidal pair $(\bL,\la)$ below $(\bG^F,s)$, Jordan
   decomposition commutes with $R_\bL^\bG$ up to twins.
  \item Let $\bL$ be an $F$-stable Levi subgroup of $\bG$ such that $s\in\bL^*$
   and let
   $\la\in\cE(\bL^F,s)$. Then $(\bL,\la)$ is an $e$-Jordan cuspidal pair of
   $\bG$ if and only if $(\bL,\la)$ is an $e$-cuspidal pair of $\bG$. Further,
   if $(\bL,\la)$ is an $e$-Jordan cuspidal pair of $\bG$, then $(\bL,\la)$ is
   $e$-Jordan quasi-central cuspidal if and only if $\la$ is of quasi-central
   $\ell$-defect.
 \end{enumerate}
\end{thm}

We will explain how this can be derived from the results of \cite{KM} in
Section~\ref{subsec:3.1 and 3.2}. The definition of twins will be given in
Section~\ref{subsec:t-twins}.

We also need the following extension of Theorem~A(a) and~(b) and Theorem~3.4
of \cite{KM15}:

\begin{thm}   \label{thm:[KM15, Thm A]}
 Let $\bX$ be as above, and $\bG\le\bX$ an $F$-stable Levi subgroup.
 \begin{enumerate}[\rm(a)]
  \item For any $e$-split Levi subgroup $\bM$ of $\bG$ and any $\ell$-block $c$
   of $\bM^F$, there exists a block $b$ of $\bG^F$ such that for every
   $\mu\in\Irr(c)\cap\cE(\bM^F,\ell')$, all irreducible constituents of
   $\RMG(\mu)$ lie in $b$.
 \item For any $e$-Jordan-cuspidal pair $(\bL,\la)$ of $\bG$ such that
   $\la\in\cE(\bL^F,\ell')$, there exists a unique $\ell$-block
   $b_{\bG^F}(\bL,\la)$ of $\bG^F$ such that all irreducible constituents
   of $\RLG(\la)$ lie in $b_{\bG^F}(\bL,\la)$.
  \item The map $\Xi:(\bL,\la)\mapsto b_{\bG^F}(\bL,\la)$ induces a
   surjection from the set of $\bG^F$-classes of $e$-Jordan-cuspidal
   pairs $(\bL,\la)$ of $\bG$ with $\la\in\cE(\bL^F,\ell')$ to the set
   of $\ell$-blocks of~$\bG^F$.
  \item The map $\Xi$ restricts to a surjection from the set of
   $\bG^F$-classes of $e$-Jordan quasi-central cuspidal pairs
   $(\bL,\la)$ of $\bG$ with $\la \in \cE(\bL^F,\ell')$ to the set
   of $\ell$-blocks of~$\bG^F$.
  \end{enumerate} 
\end{thm} 

Again, the proof will be given in Section~\ref{subsec:3.1 and 3.2}.
For future use, we also note:

\begin{prop}   \label{prop:one block}
 Let $\bX,\bG$ be as above and let $s\in\bG^{*F}$ be a semisimple
 $\ell'$-element. If there is a unique class of unipotent $e$-cuspidal pairs of
 $\bC:=C_{\bG^*}(s)$ of central defect, then $\cE_\ell(\bG^F,s)$ is a single
 $\ell$-block. In particular, if $\ell=2$ and $\bC$ has only components of
 classical type then $\cE_2(\bG^F,s)$ is a single $2$-block.
\end{prop}

\begin{proof}
Suppose that there is a unique class of unipotent $e$-cuspidal pairs of
central defect of $\bC=C_{\bG^*}(s)$. Then by Lemma~\ref{lem:qcisc} and
Proposition~\ref{prop:En 15}, there is a unique $\bG^F$-class of $e$-Jordan
quasi-central cuspidal pairs below $(\bG^F,s)$.
By Theorem~\ref{thm:[KM15, Thm A]}(d), for every $\ell$-block $b$ in
$\cE_\ell(\bG^F,s)$, there is a $\bG^F$-class of $e$-Jordan quasi-central
cuspidal pairs $(\bL,\la)$ of $\bG^F$ with $\la\in\cE(\bL^F,\ell')$ such that
$b=b_{\bG^F}(\bL,\la)$. Since $b_{\bG^F}(\bL,\la)$ contains the irreducible
constituents of $\RLG(\la)$ and since Lusztig induction preserves Lusztig
series (see, e.g., \cite[Prop.~3.3.20]{GM20}), the $\bG^F$-class of
$(\bL,\la)$ lies below~$s$. This proves the first assertion. If $\ell=2$ and
all components of $\bC$ are of classical type, the principal block is the only
unipotent $2$-block of $\bC^F$ (see for instance \cite[Thm~21.14]{CE})).
Hence, by Theorems~A and~A.bis of \cite{En00} there is only one $\bC^F$-class
of unipotent $e$-cuspidal pairs of $\bC$ of central $2$-defect.
\end{proof}

%%%%%%%%%%%%%%%%%%%%%%%%%%%%%%%%%%%%%
\subsection{$e$-Cuspidal pairs below $(\bG^F,s)$}

In this subsection and elsewhere, we will freely use the fact if $\bG$ is as
above, $s\in\bG^*$ is an $ \ell'$-element and $t\in C_{\bG^*}(s)$ is an
$\ell$-element, then $C_{\bG^*}(st) = C_{C_{\bG^*}(s)}(t)$, and therefore also
that $C_{C_{\bG^*}(s)}(t)$ is connected if in addition $s$ is semisimple.

\begin{lem}   \label{lem:isLevi}
 Let $\bX,\bG$ be as above, $s\in\bG^{*F}$ a semisimple $\ell'$-element and
 assume that $[\bX,\bX]$ is of exceptional type. Then for any $\ell$-element
 $t\in C_{\bG^*}(s)^F$, if $(\bL_t^*,\la_t)$ is a unipotent $e$-cuspidal pair
 of $C_{\bG^*}(st)$ then $\bL_t^*$ is a Levi subgroup of $C_{\bG^*}(s)$.
\end{lem}

\begin{proof}
We argue by induction on $\dim\bG$. Let $\bG_1^*\le\bG^*$ be a minimal
($F$-stable) Levi subgroup containing $C_{\bG^*}(s)$. If $\bG_1^*$ is proper, 
then by induction $\bL_t^*$ is a Levi
subgroup of $C_{\bG_1^*}(s)$ and thus of $C_{\bG^*}(s)$. Thus we may assume $s$
is isolated. Moreover, we may assume that $\ell$ is bad for $C_{\bG^*}(s)$,
since otherwise $C_{\bG^*}(st) $ is already a Levi subgroup of $C_{\bG^*}(s)$
(see \cite[Prop.~13.16]{CE}) whence so is $\bL_t^*$. Then from the list of
isolated elements (see \cite[Prop.~4.9 and Tab.~3]{B05} it transpires that for
$s$ non-central only two configurations in $\bG$ of type $E_8$ remain to be
considered:
either $\ell=3$ and $C_{\bG^*}(s)$ is of type $E_7A_1$, or $\ell=2$ and
$C_{\bG^*}(s)$ is of type $E_6A_2$. For these, we may conclude using the list
of unipotent $e$-cuspidal pairs (see \cite[Tab.~1]{BMM}). If $s$ is central,
again by induction we may assume $t$ is isolated in~$\bG^*$. Again from the
list of isolated elements, no case arises.
\end{proof}

Following Enguehard \cite{En00} we introduce a relationship between
unipotent $e$-cuspidal pairs.

\begin{defn} \label{def:sim}
 Let $\bH$ denote a connected reductive group with a Frobenius map $F$ with
 respect to an $\FF_q$-structure and let $\bH'$ be an $F$-stable
 connected reductive subgroup of $\bH$ of maximal rank. Let $(\bL,\la)$,
 $(\bL',\la')$ be unipotent $e$-cuspidal pairs of $\bH$, $\bH'$
 respectively. We write
 $$(\bL',\la')\sim(\bL,\la)$$
 if $([\bL,\bL],\la|_{[\bL,\bL]^F})$ and
 $([\bL',\bL'],\la'|_{[\bL',\bL']^F})$ are $\bH^F$-conjugate.
\end{defn}

The following is a slight variation on \cite[Prop.~17]{En00} (for which no
proof was given). Note that by the table of \cite[p.~348]{En00}, $\tw3D_4[-1]$
(respectively $\phi_{2,1}$) is the unique unipotent $1$-cuspidal (respectively
$2$-cuspidal) character of $\tw3D_4(q)$ of quasi-central $3$-defect.  

\begin{prop}   \label{prop:En 17}
 Let $\bX$ and $\bG$ be as above. Let $s\in\bG^{*F}$ be a semisimple
 $\ell'$-element and let $\bC=C_{\bG^*}(s)$. Let $t\in\bC^F$ be an
 $\ell$-element and let $(\bL_t,\la_t)$ be a unipotent $e$-cuspidal pair in
 $C_\bC(t)$ of quasi-central $\ell$-defect.
  %(so $\bL_t$ is a Levi subgroup of $\bC$ by Lemma~{\rm\ref{lem:isLevi}}).
 \begin{enumerate}[\rm(a)]
  \item If there exists an $e$-split Levi subgroup $\bM$ of $\bC$ with 
   $[\bM,\bM] =[\bL_t,\bL_t]$, then there exists a (unique) $\bC^F$-class of
   unipotent $e$-cuspidal pairs $(\bL,\la)$ in $\bC$ with
   $(\bL_t,\la_t)\sim(\bL,\la)$ as in Definition~$\ref{def:sim}$.
 \item If there exists no $e$-split Levi subgroup $\bM$ of $\bC$ with 
  $[\bM,\bM] =[\bL_t,\bL_t]$, then $\ell=3$ (so $e\in\{1,2\}$),
  $[\bL_t,\bL_t]^F$ is of type $\tw3D_4$ and there exists a unique
  $\bC^F$-class of $e$-split Levi subgroups $\bL$ of~$\bC$ with
  $[\bL,\bL]^F=D_4(q)$. Define a unipotent $e$-cuspidal character $\la$ of
  $\bL^F$ by
  \begin{enumerate}[$\bullet$]
   \item $\la=D_4$ when $e=1$ and $\la_t= \tw3D_4[-1]$,
   \item $\la=\phi_{13,02}$ when $e=2$ and $\la_t= \phi_{2,1}$.
  \end{enumerate}
 \end{enumerate}
 All pairs $(\bL,\la)$ are also of quasi-central $\ell$-defect.
\end{prop}

\begin{defn}   \label{def:-->t}
 In either case of Proposition~\ref{prop:En 17} we write
 $(\bL_t,\la_t)\to_t(\bL,\la)$. 
\end{defn}

\begin{proof}
Suppose that $(\bL,\la)$ and $(\bL',\la')$ are unipotent $e$-cuspidal pairs of
$\bC$ such that $([\bL,\bL],\la|_{[\bL,\bL]^F})$ and
$([\bL',\bL'],\la'|_{[\bL',\bL']^F})$ are $\bC^F$-conjugate. We claim that
$(\bL,\la)$ and $(\bL',\la')$ are $\bC^F$-conjugate. Indeed, let $x\in\bC^F$
with
$([\bL',\bL'],\la'|_{[\bL',\bL']^F})=\tw{x}([\bL,\bL],\la|_{[\bL,\bL]^F})$.
By \cite[Prop.~1.7(iii)]{CE94}, there
exists $c\in C_\bC^\circ(\tw{x}\bL\cap\bL')^F$ with $\bL' =\tw{cx}\bL$. Since
$[\bL',\bL']=\tw{x}[\bL,\bL] \leq\tw{x}\bL\cap\bL'$, we have
$\tw{cx}\la|_{[\bL',\bL']^F}=\tw{x}\la|_{[\bL',\bL']^F} =\la'|_{[\bL',\bL']^F}$,
and consequently by \cite[Prop.~3.1]{CE94}, $\tw{cx} \la =\la'$. It follows
that $(\bL',\la')= \tw{cx}(\bL,\la)$, proving the claim. We note that this
argument is essentially lifted from the discussion after Definition~3.4 of
\cite{CE94}.
\par
Now, suppose that we are in case (a). Let $\bL$ be an $e$-split Levi subgroup
of $\bC$ with $[\bL,\bL]= [\bL_t,\bL_t]$ and let $\la$ be the (unique)
unipotent character of $\bL^F$ with
$\la|_{[\bL,\bL]^F}= \la_t|_{[\bL_t, \bL_t]^F}$. Then $\la$ is $e$-cuspidal
(see \cite[Prop.~3.1]{CE94} and the paragraph following it) and clearly
$(\bL,\la) \sim (\bL_t, \la_t)$. The uniqueness assertion of~(a) follows from
the paragraph above.
\par 
Suppose that $\ell$ is odd, good for $\bC$ and $\ell>3$ if $\bC^F$ has no
component of type $\tw3D_4$. Then we are in case~(a) by \cite[Prop.~3.5]{CE94}.
Note that in loc.~cit.~the element~$t$ is in the dual of the group about which
the assertion is being made, but one can check that the proof works exactly in
the same way in the situation we are considering. 
\par 
Suppose that $\ell=3$ and $\bC^F$ has a component of type $\tw3D_4$. Then by
Proposition~\ref{prop:3D4}, $\bX$ is of exceptional type or $\bX^F$ is of type
$\tw3D_4$, and by rank
considerations $\bC$ has a single component of type $D_4$ and all other
components of $\bC$ are of type $A$. Write $\bC =\bC_1\bC_2$ where $\bC_1$ has
type $D_4$ and $\bC_2$ is the product of all other components of $[\bC,\bC]$
with $Z^\circ(\bC)$. Since $Z(\bC_1)$ is a $2$-group, $t=t_1t_2$ with
$t_i\in\bC_i ^F$ and $C_\bC(t)=C_{\bC_1}(t_1)C_{\bC_2}(t_2)$. Moreover
$\bL_t = \bM_1\bM_2$ with each $\bM_i$ being $e$-split in $C_{\bC_i}(t_i)$, and
$\la_t$ covers the irreducible character $\mu_1\mu_2$ of
$C_{\bC_1}(t_1)^F C_{\bC_2}(t_2)^F$ where $\mu_i$ is a unipotent $e$-cuspidal
character of $\bM_i$ (this follows for instance from \cite[Sect.~3.1]{CE94} and
the fact that unipotent $e$-cuspidal pairs behave well under taking direct
products (see for instance Proposition~\ref{prop:JoEn 2.1.5} and note that
$e$-Jordan cuspidality and $e$-cuspidality coincide in the unipotent case).
By Lemma~\ref{lem:3D4b}, there is an $e$-split Levi subgroup of
$C_{\bC_1}(t_1)$, say $\bL_1$ with $[\bL_1,\bL_1] =[\bM_1,\bM_1]$ (we take
$\bL_1=\bM_1 =\bT$ where $\bT$ is as in the
lemma). By the argument in the preceding paragraph, there is an $e$-split Levi
subgroup of $C_{\bC_2}(t_2)$, say $\bL_2$ with $[\bL_2,\bL_2] =[\bM_2,\bM_2]$.
Then $\bL = \bL_1\bL_2$ is an $e$-split Levi subgroup of $\bC$ with
$[\bL,\bL] = [\bL_t, \bL_t]$ and we are in case~(a).
\par
Now suppose that $\ell=2$ and all components of $\bL_t$ are of classical type.
By Lemma~\ref{lem:2unipquasicentraldefect}, $\bL_t$ is a torus and we may take
$\bL$ to be the centraliser of the corresponding Sylow $e$-torus of $\bC^F$
(see \cite[Lemma~3.17]{KM15}).
\par
In the remaining cases, recall that we need to prove the following: if either
$\ell\ne 3$ or $[\bL_t,\bL_t]^F$ is not of type $\tw3D_4$, then there exists an
$e$-split Levi subgroup $\bL$ of $\bC$ with $[\bL,\bL]=[\bL_t,\bL_t]$, and that
if $\ell=3$ and $[\bL_t,\bL_t]^F$ is of type $\tw3D_4$, then $\bC$ has an
$e$-split Levi subgroup $\bL$ with $[\bL,\bL]^F$ of type $D_4$. If $\bC$ is
contained in a proper Levi subgroup $\bG_1$ of~$\bG^*$
then by induction on $\dim\bG$ there exists an $e$-split Levi subgroup
$\bL_{\bG_1}$ of $C_{\bG_1^*}(s)=\bC$ as wanted, and we can take
$\bL:=\bL_{\bG_1}$ in $\bC$. Hence, we may assume that $s$ is isolated
in~$\bG^*$. \par
Let $\bG_1:=C_{\bG^*}(Z(\bL_t)_e)$, an $e$-split Levi subgroup of $\bG^*$
containing $\bL_t$. Then, $C_{\bG_1} (s) $ is $e$-split in $\bC$. If $\bG_1$ is
proper in~$\bG^*$, then by induction on $\dim \bG$ there exists an
$e$-split Levi subgroup $\bL_{\bG_1}$ of $C_{\bG_1}(s)$, (which is then also
$e$-split in~$\bC$) as wanted, and we can take $\bL= \bL_{\bG_1}$.
Now assume $\bG_1=\bG^*$. Then since $(\bL_t,\la_t)$ is $e$-cuspidal in
$C_\bC(t)$, so in particular $\bL_t$ is $e$-split in $C_\bC(t)$, we have
$$\bL_t=C_{C_\bC(t)}(Z(\bL_t)_e)=C_\bC(t)\cap C_{\bG^*}(Z(\bL_t)_e)
 =C_\bC(t)\cap \bG^*=C_\bC(t).$$
\par
By the discussion above, we may also
assume now that either $\ell=2$, or $\ell=3$ and $(\bC,F)$ has a factor of
exceptional type, or $\ell=5$, $\bG$ is of type $E_8$ and $s=1$. So we have
$e\in\{1,2,4\}$. Using \Chevie\ \cite{MChev} we can enumerate all possible
candidates for $\bL_t$, that is, all rational types of Levi subgroups~$\bM$ of
the various $\bC$ having $e$-cuspidal unipotent characters of quasi-central
$\ell$-defect. Here, note that by Lemma~\ref{lem:isLevi}, $\bL_t$ is a Levi
subgroup of~$\bC$. It turns out that we are in one of three cases when $e=1$
(the cases $e=2$, and $e=4$ in $E_8$, being entirely similar):
\begin{itemize}
\item The only non-trivial $\ell$-elements in $Z(\bM)^F$ are involutions, but
 $\bM$ is not the centraliser of an involution in any of the possible $\bC$.
\item $Z(\bM)_e>Z(\bC)_e$  (here, as earlier, an index $e$ on an
 $F$-stable torus denotes its Sylow $e$-subtorus) and so
 $C_{\bG^*}(Z(\bM)_e)<\bG^*$, and we may conclude by induction, see above.
\item $\bM$ has rational type $\tw3D_4(q)\Phi_3^k$ for $k=1,2$, where $Z(\bM)^F$
 contains non-trivial $\ell$-elements only when $\ell=3$.
\end{itemize}
In the last case, $\bC$ has a factor of type $E_n$, $n\ge6$, and so does possess
a 1-split Levi subgroup of type $D_4$, unique up to conjugacy, whence we end up
in case~(b). By explicit
enumeration, for $\ell=3$ whenever $(\bL_t,\la_t)$ is a unipotent
$e$-cuspidal pair in $C_\bC(t)$ such that $\bL_t$ has a factor of type
$\tw3D_4$, then $[\bL_t,\bL_t]^F$ is simple of type~$\tw3D_4$. Also, it turns
out that none of the relevant $\bC$ do possess an $e$-split Levi subgroup,
$e\in\{1,2\}$, with a component of type~$\tw3D_4$.
The statement about quasi-central defect is obvious if we are not in
situation~(b); in the latter case it can be checked directly.
\end{proof}

\begin{rem}   \label{rem:p368}
In \cite[Prop.~17]{En00} it is claimed that if one of $(\bL_t,\la_t)$,
$(\bL,\la)$ is of central $\ell$-defect, then so is the other. This is easily
seen to not hold in two of the four exceptional cases listed there.
It is also stated that $\bL_t$ is ``deploy\'e'' (split), but this seems to be a
misprint as it can easily be seen to be wrong in general. Let us add that
neither is $\bL_t$ an $e$-split Levi subgroup of $\bC$, in general.
\end{rem}

%%%%%%%%%%%%%%%%%%%%%%%%%%%%%%%%%%%%%%%%%%%%%%%%%%%%%%%%%%%%%%%%%%%%%%%%%
\section{The block distribution}   \label{sec:blocks}
In this section we prove Theorem~\ref{thm:thmA}. Several of the results
presented here, or variants thereof, were stated in or are inspired by the
work of Enguehard \cite{En00} on unipotent blocks for bad primes. Since
\cite{En00} does not always give proofs, and some of its statements are
obviously inaccurate, we have decided to provide some of the missing proofs
(and indicate where we believe \cite{En00} is incorrect). We take the
opportunity to point out that, in particular, Theorems~A and A.bis of
\cite{En00} seem not correct for $\ell=2$ when $e$ is \emph{not} the order of
$q$ modulo~4.
\medskip

Throughout this section, we fix the following notation:
\begin{itemize}
\item Let $\bX$ be a connected reductive group in characteristic~$p$ with
 connected centre and simply connected simple derived subgroup with a
 Frobenius map $F:\bX\rightarrow\bX$, and let $\bG$ be an $F$-stable Levi
 subgroup of $\bX$;
\item let $\ell$ be a prime not dividing~$q$ and $e:=e_\ell(q)$; and
\item let $s\in\bG^{*F}$ be a semisimple $\ell'$-element and set
$\bC^*:=C_{\bG^*}(s)$.
\end{itemize}
Note that $\bC^*$ is connected, as all centralisers of semisimple elements
in~$\bG^*$, since $\bG$ has connected centre. For any $\ell$-element
$t\in\bC^{*F}$ recall the Digne--Michel Jordan decomposition
$$\pi_{st}^\bG:\cE(\bG^F,st)\longrightarrow\cE(C_{\bG^*}(st)^F,1)
  =\cE(C_{\bC^*}(t)^F,1)$$
already discussed in Section~\ref{sec:general}. For $(\bL,\la)$ an $e$-Jordan
cuspidal pair of $\bG$ below $(\bG^F,s)$ we write $b_{\bG^F}(\bL,\la)$ for the
corresponding $\ell$-block of $\bG^F$ containing all constituents of
$\RLG(\la)$ (see Theorem~\ref{thm:[KM15, Thm A]}).

%%%%%%%%%%%%%%%%%%%%%%%%%%%%%%%%%%%%%
\subsection{The map $\bJ_t^\bG$}
We start by defining the map $\bJ_t^\bG$ in Theorem~\ref{thm:thmA}, based on
the relation $\to_t$ introduced in Proposition~\ref{prop:En 17} (see
Definition~\ref{def:-->t}).

\begin{prop}   \label{prop:J_t}
 Let $t\in\bC^{*F}$ be an $\ell$-element.
 \begin{enumerate}[\rm(a)]
  \item The relationship $\to_t$ on $e$-cuspidal pairs defined in
   Proposition~$\ref{prop:En 17}$ induces a map $J_t^\bG$ from the set of
   $C_{\bG^*}(st)^F$-classes of unipotent $e$-cuspidal pairs of quasi-central
   $\ell$-defect in $C_{\bG^*}(st)$ to the set of $\bG^F$-classes of $e$-Jordan
   quasi-central cuspidal pairs in $\bG$ below $(\bG^F,s)$. 
  \item This induces a map $\bJ_t^\bG$ from the set of unipotent $\ell$-blocks
   of $C_{\bG^*}(st)^F$ to the set of $\ell$-blocks of $\bG^F$ in
   $\cE_\ell(\bG^F,s)$.
 \end{enumerate}
\end{prop}
 
\begin{proof}
Let $(\bL_t^*,\la_t)$ be a unipotent $e$-cuspidal pair of quasi-central
$\ell$-defect in $\bC_t^*:=C_{\bG^*}(st)=C_{\bC^*}(t)$. By
Proposition~\ref{prop:En 17} the $\bC_t^{*F}$-class of $(\bL_t^*,\la_t)$ gives
rise to a unique $\bC^{*F}$-class of unipotent $e$-cuspidal pairs
$(\bL_s^*,\la_s)$ in $\bC^*$, of quasi-central defect. Then the bijection in
Proposition~\ref{prop:En 15} provides a $\bG^F$-class of $e$-Jordan
quasi-central cuspidal pairs $(\bL,\la)$ in $\bG$ below $(\bG^F, s)$ and we
define the image of the class of $(\bL_t^*,\la_t)$ under $J_t^\bG$ to be the
class of $(\bL,\la)$.   \par
For (b) let $b$ be a unipotent $\ell$-block of $\bC_t^{*F}$. By
\cite[Thm~A and~A.bis]{En00} there
exists a unipotent $e$-cuspidal pair $(\bL_t^*,\la_t)$ in $\bC_t^*$, of
quasi-central $\ell$-defect and unique up to $\bC_t^{*F}$-conjugacy such that
$b=b_{\bC_t^{*F}}(\bL_t^*,\la_t)$. The map $J_t^\bG$ from~(a) provides an
$e$-Jordan quasi-central cuspidal pair $(\bL,\la)$ in $\bG$ below $(\bG^F,s)$,
unique up to $\bG^F$-conjugacy, and by
Theorem~\ref{thm:[KM15, Thm A]} this determines an $\ell$-block
$\bJ_t^\bG(b):=b_{\bG^F}(\bL,\la)$ of $\bG^F$ in series $s$.
\end{proof}

We note that when $s=1 $, the map $\bJ_t^{\bar G}$ above coincides with the
map $\bJ_t^{\bG, F}$ of \cite[Thm~B]{En00} (where our semisimple $\ell$-element
$t$ is denoted $s$).

\begin{rem}
In the setting of Proposition~$\ref{prop:J_t}$ let $(\bL_t^*, \la_t)$ be a
unipotent $e$-cuspidal pair in $C_{\bG^*}(st)=C_{\bC^*}(t)$. If
$\bL_t^*$ is an $e$-split Levi subgroup of $\bG^*$ with dual $\bL_t\leq\bG$,
then by Proposition~\ref{prop:En 17}(a), $J_t^\bG$ sends the
$C_{\bG^*}(st)^F$-class of $(\bL_t^*,\lambda_t)$ to the $\bG^F$-class of
$(\bL_t,\lambda')$ for some $\la'\in\cE(\bL_t^F,s)$.
\end{rem}

In the following, by \emph{semisimple block} in the union of Lusztig series
$\cE_\ell(\bG^F,s)$ we mean the block containing the semisimple character of
$\cE(\bG^F,s)$. 
Note that Lusztig induction of a
semisimple character contains a semisimple character. This follows from the
fact that Jordan decomposition preserves uniform functions and the fact that
the trivial character is a constituent of any Lusztig induction of the trivial
character (see e.g.~\cite[proof of Cor.~10.1.7]{DM20}).

\begin{prop}   \label{prop: ss}
 Let $t\in C_{\bG^*}(s)^F$ be an $\ell$-element and $b$ the principal
 $\ell$-block of $C_{\bG^*}(st)^F$. Then $\bJ_t^\bG(b)$ is the semisimple
 block in $\cE_\ell(\bG^F,s)$.
\end{prop}

\begin{proof}
The principal block of $C_{\bG^*}(st)^F$ is labelled by the $e$-Harish-Chandra
series of $(\bL_t^*,1)$ for $\bL_t^*$ the centraliser of a Sylow $e$-torus of
$C_{\bG^*}(st)$. Thus we are not in one of the exceptional cases of
Proposition~\ref{prop:En 17} whence the corresponding unipotent $e$-cuspidal
pair of $C_{\bG^*}(s)$ is of the form $(\bL_s^*,1)$.
Let $(\bL,\la)$ be associated to $(\bL_s^*,1)$ as in
Proposition~\ref{prop:En 15}, so $\la$ is the semisimple character in
$\cE(\bL^F,s)$. Hence the semisimple block in $\cE_\ell(\bG^F,s)$ lies above
$(\bL,\la)$, so equals $\bJ_t^\bG(b)$.
\end{proof}

By Theorem~\ref{thm:[KM15, Thm A]}, for any $e$-split Levi $\bM$ of $\bG$
there is a map $\RMG$ from the set of $\ell$-blocks in $\cE_\ell(\bM^F,s)$ to
the set of $\ell$-blocks in $\cE_\ell(\bG^F,s)$ such that if $c$ is an
$\ell$-block of $\bM^F$ in series~$s$, then every constituent of $\RMG(\mu)$
for any $\mu\in\Irr(c)\cap\cE(\bM^F,s)$, lies in $\RMG(c)$. 

The next statement mirrors \cite[Prop.~16]{En00}:

\begin{prop}   \label{prop:En 16}
 Let $\bM\le\bG$ be $e$-split with $s\in\bM^*$. Let $(\bL_i,\la_i)$, $i=1,2$,
 be two $e$-Jordan cuspidal pairs below $(\bM^F,s)$. Then these are
 $e$-Jordan cuspidal pairs below $(\bG^F,s)$, and if
 $b_{\bM^F}(\bL_1,\la_1)=b_{\bM^F}(\bL_2,\la_2)$ then also
 $b_{\bG^F}(\bL_1,\la_1)=b_{\bG^F}(\bL_2,\la_2)$.   \par
 Further, if $(\bL,\la)$ is an $e$-Jordan quasi-central cuspidal pair below a
 block $b$ in $\cE_\ell(\bM^F,s)$ then it is so below $\RMG(b)$.
\end{prop}

\begin{proof}
The first statement is clear from the definition of $e$-Jordan cuspidal pairs
and the fact that if $\bL$ is $e$-split in $\bM$ and $ \bM$ is $e$-split in
$\bG$, then $\bL$ is $e$-split in $\bG$. Suppose that
$b_{\bM^F}(\bL_1,\la_1)=b_{\bM^F}(\bL_2,\la_2)=:c$. Then every irreducible
constituent of $R_{\bL_i}^\bM (\la_i)$, $i=1,2$, lies in $c$. On the other
hand, if $\chi$ is an irreducible constituent of $R_{\bL_i}^\bG(\la_i)$ for
any $i=1,2$, then $\chi$ is a constituent of $\RMG(\mu_i)$ for some irreducible
constituent $\mu_i$ of $R_{\bL_i}^\bM (\la_i)$ and as observed before $\mu_i$
belongs to $\Irr(c)$. This proves the result.
\end{proof}

The next statement extends \cite[Cor.~19]{En00}.

\begin{cor}   \label{cor:J_t}
 Let $t\in \bC^{*F}$ be an $\ell$-element. If $\bM^*\le\bG^*$ is an $e$-split
 Levi subgroup such that $C_{\bG^*}(st)\le\bM^*$, with dual $\bM\le\bG$, then
 $\RMG\circ \bJ_t^\bM=\bJ_t^\bG$ on the set of unipotent $\ell$-blocks of
 $C_{\bM^*}(st)^F$.
\end{cor}

\begin{proof} 
Let $c$ be a unipotent $\ell$-block of $C_{\bG^*}(st)$ and $(\bL_t,\la_t)$ be a
unipotent $e$-cuspidal pair of quasi-central $\ell$-defect defining $c$. Let
$(\bL',\la')$ be a unipotent $e$-cuspidal pair (also of quasi-central
$\ell$-defect) of $C_{\bM^*}(s) $ such that $(\bL_t,\la_t)\to_t (\bL',\la')$
in $\bM^*$.
Then, $C_{\bM^*}(s)$ is $e$-split in $\bC^*$, hence $(\bL',\la')$ is an
$e$-cuspidal pair of $\bC^*$ and $(\bL_t,\la_t)\to_t (\bL',\la')$ in $\bG^*$
by Proposition~\ref{prop:En 17}. Further, since
$Z^\circ(\bM^*)_e\leq Z^\circ(C_{\bM^*}(s))_e\leq Z^\circ(\bL')_e$ and
$\bM^*=C_\bG^*(Z^\circ(\bM^*)_e)$, it follows that
$C_{\bG^*}(Z^\circ(\bL')_e) = C_{\bM^*}(Z^\circ(\bL')_e)$.
So $(\bL,\la) :=J_t^\bG (\bL_t,\la_t) = J_t^\bM(\bL_t,\la_t)$.
Now it follows from Proposition~\ref{prop:En 16} that
\begin{equation}
  \bJ_t^\bG(c)=b_{\bG^F}(\bL,\la)=\RMG(b_{\bM^F}(\bL',\la'))=\RMG(\bJ_t^\bM(c)).
  \qedhere\end{equation}
\end{proof}

%%%%%%%%%%%%%%%%%%%%%%%%%%%%%%%%%%%%%
\subsection{$t$-Twin blocks}   \label{subsec:t-twins}
Our main theorem relies on the compatibility between Lusztig induction and
Jordan decomposition. In order to deal with certain ambiguities in groups of
type $E_8$, we group together certain unipotent $2$-cuspidal pairs of
exceptional groups into \emph{twins} as follows:
\begin{itemize}
 \item $(E_6,E_6[\theta])$ and $(E_6,E_6[\theta^2])$,
 \item $(\tw2E_6,\tw2E_6[\theta])$ and $(\tw2E_6,\tw2E_6[\theta^2])$,
 \item $(E_7,\phi_{512,11})$ and $(E_7,\phi_{512,12})$.
\end{itemize}
Now for any Levi subgroup of $\bG$ having one of the above as a component (note
that there can be at most one such component), we call \emph{twin
$2$-Harish-Chandra series} the corresponding unions of $2$-Harish-Chandra
series of $\bG$ lying above these (see also \cite[p.~350]{GM20}). Note that the
two members in each of the first two pairs are Galois conjugate over
$\QQ(\theta)$,
for $\theta$ a primitive third root of unity, while those in the last one are
Galois conjugate over $\QQ(\sqrt{q})$ when $q$ is not a square, and rational
otherwise.

For $s\in\bG^{*F}$ a semisimple $\ell'$-element, we define \emph{$t$-twin
blocks}, for $t\in C_{\bG^*}(s)$ an
$\ell$-element, as follows: If $\bG$ is not of type $E_8$, or if
$e_\ell(q)\ne2$, then $t$-twin blocks are blocks. Suppose that $\bG=E_8$ and
$e_\ell(q)=2$. Let $(\bL_t^*,\la_t)$ be a unipotent $e$-cuspidal
pair of $C_{\bG^*}(st) = C_{\bC^*}(t)$ of quasi-central $\ell$-defect and let
$b=\bJ_t^\bG(b_{C_{\bG^*}(st)^F}(\bL_t^*,\la_t))$. Then the $t$-twin block
containing $b$ is the pair consisting of $b$ and its ``twin" corresponding to
the twin $(\bL_t^*,\la_t')$ of $(\bL_t^*,\la_t)$ if $(\bL_t^*,\la_t)$ is as
above; see the list shown in Table~\ref{tab:twins}. Otherwise, the $t$-twin
block of $b$ contains only~$b$. With this, the map $\bJ_t^\bG$ from
Proposition~\ref{prop:J_t} can and will be considered as a map from the set of
unipotent blocks of $C_{\bG^*}(st)^F$ to the set of $t$-twin blocks of $\bG^F$.

\begin{table}[htbp]
\caption{$t$-Twin blocks, $\ell|(q+1)$}   \label{tab:twins}
$$\begin{array}{|c|cc|cc|}
\hline
 \bG& C_{\bG^*}(s)^F& \ell& \bL_t^{*F}& \la_t,\la_t'\\
\hline\hline
% E_7& \bG^{*F}& 3& \Ph2.\tw2E_6(q) &\tw2E_6[\theta],\tw2E_6[\theta^2]\\
%\hline
 E_8& \bG^{*F}& \ne5& \Ph2^2.\tw2E_6(q)& \tw2E_6[\theta],\tw2E_6[\theta^2]\\
 E_8& \tw2E_6(q).\tw2A_2(q)& \ne3& \Ph2^2.\tw2E_6(q)& \tw2E_6[\theta],\tw2E_6[\theta^2]\\
 E_8& E_7(q).A_1(q)& \ne2& \Ph2^2.\tw2E_6(q)& \tw2E_6[\theta],\tw2E_6[\theta^2]\\
\hline
 E_8& \bG^{*F}& \ne2,3& \Ph2.E_7(q)& \phi_{512,11},\phi_{512,12}\\
 E_8& E_7(q).A_1(q)& \ne2& \Ph2.E_7(q)& \phi_{512,11},\phi_{512,12}\\
\hline
\end{array}$$
\end{table}

%%%%%%%%%%%%%%%%%%%%%%%%%%%%%%%%%%%%%
\subsection{Inductive arguments}
The following statement will allow us to inductively deal with many cases of
Theorem~\ref{thm:thmA}.

\begin{prop}   \label{prop:in Levi}
 Let $t\in \bC^{*F}$ be an $\ell$-element and $\chi\in\cE(\bG^F,st)$.
 Assume that $C_{\bG^*}(st)$ is contained in an $e$-split proper Levi subgroup
 $\bM^*$ of $\bG^*$ and that Theorem~$\ref{thm:thmA}$ holds for its dual~$\bM$.
 Let $b_t$ be the unipotent $\ell$-block of $C_{\bG^*}(st)^F$ containing
 $\pi_{st}^\bG(\chi)$. Then $b_{\bG^F}(\chi)$ is contained in $\bJ_t^\bG(b_t)$.
\end{prop}

\begin{proof}
By Lemma~\ref{lem:in Levi}, there is $\gamma\in\cE(\bG^F,s)$ with
$\langle d^{1,\bG}(\chi),\gamma\rangle\ne0$, so in particular lying in the
$\ell$-block $b_{\bG^F}(\chi)=b_{\bG^F}(\gamma)$. Let $\bM\le\bG$ be dual to
$\bM^*$ and let $\chi'\in\cE(\bM^F,st)$ with $\chi=\pm\RMG(\chi')$
and $\pi_{st}^\bG(\chi)=\pi_{st}^\bM(\chi')$ (see \cite[Thm~7.1]{DM90}). Then
$$\blangle d^{1,\bM}(\chi'),\sRMG(\gamma)\brangle
  =\blangle d^{1,\bG}(\RMG(\chi')),\gamma\brangle
  =\pm\langle d^{1,\bG}(\chi),\gamma\rangle\ne0,$$
whence there is a constituent $\gamma'\in\cE(\bM^F,s)$ of $\sRMG(\gamma)$ with
$$\blangle d^{1,\bM}(\chi'),\gamma'\brangle
  \blangle \sRMG(\gamma),\gamma'\brangle \ne0.$$
By assumption and since $\blangle d^{1,\bM}(\chi'),\gamma'\brangle\ne0$ we know
that $b_{\bM^F}(\gamma')=b_{\bM^F}(\chi')$ is contained in 
 $\bJ_t^\bM(b_t)$. Application of Proposition~\ref{prop:En 16} and
Corollary~\ref{cor:J_t} then yields
$$b_{\bG^F}(\chi)=b_{\bG^F}(\gamma)=\RMG(b_{\bM^F}(\gamma'))\quad
  \text{lies in}\quad\RMG(\bJ_t^\bM(b_t))=\bJ_t^\bG(b_t),$$
as claimed.
\end{proof}

Another useful reduction is the following:

\begin{prop}   \label{prop:isolated reduction}
 Let $\bM$ be a proper $F$-stable Levi subgroup of $\bG$ whose dual in $\bG^*$
 contains $C_{\bG^*}(s)$ and let $ t \in C_{\bG^*}(s)^F$ be an $\ell$-element.
 If Theorem~$\ref{thm:thmA}$ holds for $\bM$ and $t$, then it
 holds for $\bG$ and $t$.
\end{prop}

\begin{proof}
Recall that by the results of Bonnaf\'e--Rouquier, $\pm\RMG$ induces a
bijection (referred to as Jordan correspondence in \cite{KM15}),
which we will abbreviate \emph{BR-correspondence} here, between the
$\ell$-blocks in $\cE_\ell(\bM^F,s)$ and the $\ell$-blocks in
$\cE_\ell(\bG^F,s)$. By \cite[Prop.~2.4, 2.6 and Thm~A]{KM15}, there exists a
bijection $\Psi_\bM^\bG$ between the set of
$e$-Jordan cuspidal pairs below $(\bM^F,s)$ and the set of $e$-Jordan cuspidal
pairs below $(\bG^F,s)$ such that for any $e$-Jordan cuspidal pair $(\bL',\la')$
below $(\bM^F,s)$, $b_{\bG^F}(\Psi_\bM^\bG(\bL',\la'))$ is the BR-correspondent
of $b_{\bM^F} (\bL',\la')$. The bijection $\Psi_\bM^\bG$ is described as
follows: If $(\bL',\la')$ is an $e$-Jordan cuspidal pair below $(\bM^F,s)$,
then $\Psi_\bM^\bG(\bL',\la')=(\bL,\la)$, where
$\bL= C_{\bG} (Z^\circ (\bL')_e)$ and $\la = \pm R_{\bL'}^\bL(\la')$. 

Passing to duals we have that $\bL^*= C_{\bG^*} (Z^\circ ({\bL'}^{*})_e)$.
Since ${\bL'}^*$ is $e$-split in $\bM^*$, ${\bL'}^*
=C_{\bM^*}(Z^\circ ({\bL'}^*)_e)$ and hence ${\bL'}^* = \bL^*\cap\bM^*$.
Consequently, $C_{\bL^*}(s) = C_{{\bL'}^*}(s)$ and by properties of
Digne--Michel's Jordan
decomposition, $\pi_s^\bL(\la)=\pi_s^{\bL'}(\la')$. Thus, if $(\bL',\la')$
corresponds to the unipotent $e$-cuspidal pair $(\bL_s^*,\la_s)$ of
$C_{\bM^*}(s)$ via Proposition~\ref{prop:En 15} (for the group $\bM$), then
$(\bL,\la)$ also corresponds to $(\bL_s^*,\la_s)$ via
Proposition~\ref{prop:En 15}. 

Let $t \in C_{\bG^*}(s)^F$ be an $\ell$-element. It follows from the above that
for any unipotent $e$-cuspidal pair $(\bL_t^*,\la_t)$ of $C_{\bG^*}(st)$, if
$(\bL',\la')$ is in the $\bM^F$-class of $J_t^\bM((\bL_t^*,\la_t))$, then
$(\bL,\la)$ is in the $\bG^F$-class of $J_t^\bG((\bL_t,\la_t))$. Consequently,
for any unipotent block $b$ of $C_{\bG^*}(st)$, $\bJ_t^\bG(b)$ is the
BR-correspondent block of $\bJ_t^\bM(b)$.
 
Now let $\chi\in\cE(\bG^F,st)$. Since $C_{\bG^*}(st)\leq\bM^*$, there exists
$\chi'\in\cE(\bM^F,st)$ with $\chi=\pm\RMG(\chi')$ and
$\gamma:=\pi_{st}^\bG(\chi)=\pi_{st}^\bM(\chi')$ (see \cite[Thm~7.1]{DM90}).
Denoting by $b$ the block containing~$\gamma$, by hypothesis we have
$\chi'\in\Irr(\bJ_t^\bM(b))$ (here we note that as $\bM$ is proper in $\bG$,
$t$-twin blocks of $\bM$ are the same as blocks). Thus, $\chi$ belongs to
the BR-correspondent of $\bJ_t^\bM(b)$, that is to $\bJ_t^\bG(b)$.
\end{proof}

\begin{prop}   \label{prop:s central}
 The conclusion of Theorem~$\ref{thm:thmA}$ holds whenever $s\in Z(\bG^{*F})$.
\end{prop}

\begin{proof}
Let $s\in Z(\bG^{*F})$, so $\cE(C_{\bG^*}(t)^F,1)=\cE(C_{\bG^*}(st)^F,1)$. Let
$\hat s\in\Irr(\bG^F)$ be the linear character of $\ell'$-order determined
by~$s$. Then $\underline{\ \ }\otimes\hat s:\cE(\bG^F,t)\to\cE(\bG^F,st)$
preserves the partition into $\ell$-blocks. Now by Enguehard
\cite[Cor.~19, Thm~B]{En00}, the conclusion of Theorem~\ref{thm:thmA} holds
whenever $s=1$. The claim follows since tensoring with $\hat s$ commutes with
Lusztig induction (see \cite[Prop.~9.6]{CE} and note that all unipotent
elements of $\bG^F$ are contained in the kernel of any linear character
of $\bG^F$).
\end{proof}

%%%%%%%%%%%%%%%%%%%%%%%%%%%%%%%%%%%%%
\subsection{Theorem~\ref{thm:thmA} when $t$ is central}
\begin{prop}   \label{prop:several unipotents}
 Suppose that $s\in\bG^{*F}$ is isolated and non-central. Let $(\bL_s^*,\la_s)$
 be a unipotent $e$-cuspidal pair of $\bC^*$ and let $(\bL,\la)$ be an
 associated $e$-Jordan cuspidal pair below $(\bG^F,s)$ via
 Proposition~{\rm\ref{prop:En 15}}. Then $b_{\bG^F}(\bL,\la)$ is in
 $\bJ_1^\bG (b_{\bC^{*F}}(\bL_s^*,\la_s))$. 
\end{prop} 

\begin{proof} 
If $(\bL_s^*, \la_s) $ is of quasi-central $\ell$-defect then the result
follows from the definition of $J_1^\bG$ and the fact that by
Proposition~\ref{prop:En 17}(a), the relation $\to_1$ is the identity. Since
every unipotent block has an associated unipotent $e$-cuspidal pair of
quasi-central $\ell$-defect, it remains to show that if $(\bL^*_s,\la_s)$ and
$({\bL'}_s^*,\la_s')$ are two unipotent $e$-cuspidal pairs of $\bC^*$,
corresponding via Proposition~\ref{prop:En 15} to pairs $(\bL,\la)$,
$(\bL',\la')$ respectively, of $\bG^F$, and if
$b_{\bC^{*F}}(\bL_s^*,\la_s) = b_{\bC^{*F}} ({\bL'}_s^*,\la'_s)$, then
$b_{\bG^F}(\bL,\la) = b_{\bG^F}(\bL',\la')$.
\par
Let $b$ be a unipotent block of $\bC^*$ containing two distinct
unipotent $e$-Harish-Chandra series. By the main result of \cite{CE94} we have
that either $\ell=2$, or $\ell$ is bad for $\bC^*$ (since $\tw3D_4(q)$
is not a component of $\bC^{*F}$ by Lemma~\ref{prop:3D4}).
\par
Suppose first that all components of $\bG$ (and hence of $\bC^*$) are of
classical type. Then by the above $\ell=2 $ and $\cE_2(\bG^F,s)$ is a single
$2$-block (see for instance \cite[Thm~21.14]{CE}). Thus we may assume
$\bG$ has a component of exceptional type and $\ell$ is bad for $\bG$.
\par
If $[\bG,\bG] $ is simple, then the result follows from
Theorem~\ref{thm:[KM13, 1.2]}.
Thus we may assume that $[\bG,\bG]$ is not simple and therefore also that the
semisimple rank of $\bG$ is at most~$7$, as $\bX$ is of exceptional type and
$\bG$ is proper in $\bX$. Since $\bG$ has a component of
exceptional type, it follows that $[\bG,\bG]$ is of type $E_6A_1$ and $\ell=2$
or~$3$. Since groups of type~$A$ have no non-central isolated elements, the
$E_6$-component of $s$ in $\bG^*$ is non-central and isolated. Thus, by
\cite[Tab.~3]{B05}, one sees that all components of $\bC^*$ are of
type~$A$. Since $3$ is good for groups of type $A$, this forces $\ell=2$.
Then the result follows by Proposition~\ref{prop:one block}.
\end{proof}

\begin{prop}   \label{prop:GeMa 4.7.7-2}
 Suppose that $s\in\bG^{*F}$ is isolated and non-central. Let $(\bL_s^*,\la_s)$
 be a unipotent $e$-cuspidal pair of $\bC^*$ and let $(\bL,\la)$ be an
 associated $e$-Jordan cuspidal pair below~$(\bG^F,s)$ via
 Proposition~{\rm\ref{prop:En 15}}. Then there exists a bijection
 $$\tilde\pi_s^\bG:\cE(\bG^F,(\bL, \la))\rightarrow
     \cE(\bC^{*F},(\bL_s^*,\la_s))$$
 such that any constituent $\chi$ of $\RLG(\la)$ has Jordan correspondent
 $\tilde\pi_s^\bG(\chi)$, unless possibly if $\bG=E_8$, and $\bC^*$
 and $(\bL_s^*,\la_s)$ belong to a ``twin", in which case the Jordan
 correspondent could be a constituent of $R_{\bL_s^*}^{\bC^*}(\la_s')$
 for $(\bL_s^*,\la_s')$ the twin of $(\bL_s^*,\la_s)$.

 In particular, $\tilde\pi_s^\bG(\chi)$ and the Jordan correspondent of
 $\chi$ have the same degree, so $\tilde\pi_s^\bG$ preserves $\ell$-defects.
\end{prop}

\begin{proof}
Suppose first that $[\bG,\bG]^F\ne E_8(2)$. Then the Mackey formula holds
for~$\bG^F$ and the result follows by \cite[Cor.~4.7.7]{GM20}. Here, we note
that in the statement of \cite[Cor.~4.7.7]{GM20}, the assumption on the Mackey 
formula is on the ambient group, and that this group is itself simple, but it
can be easily checked that the proof and result carry over with our assumptions.
Now suppose that $[\bG,\bG]^F= E_8(2)$. This implies that there is an $F$-stable
decomposition $\bG=E_8\times\bT$ with a torus $\bT$. Using the compatibility of
Lusztig induction with direct products the result then follows from
\cite[Prop.~3.11]{Ho22} when $\ell\ge7$, and from \cite[Prop.~6.11]{KM} when
$\ell\in\{3,5\}$.
\end{proof}

\begin{prop}   \label{prop:t central}
 The conclusion of Theorem~$\ref{thm:thmA}$ holds for any $t\in Z(\bG^{*F})$.
\end{prop}

\begin{proof}
Let $t\in Z(\bG^{*F})$ and $\hat t\in\Irr(\bG^F)$ be the linear character of
$\ell$-power order determined by~$t$. Then
$\underline{\ \ }\otimes\hat t:\cE(\bG^F,s)\to\cE(\bG^F,st)$ is a bijection
preserving the partition into $\ell$-blocks, and
$\cE_\ell(C_{\bG^*}(s)^F,1)=\cE_\ell(C_{\bG^*}(st)^F,1)$ so it suffices to
prove the result for $t=1$. 

Suppose that $t=1$. By Proposition~\ref{prop:isolated reduction} we may assume
$s$ is isolated in $\bG^*$ and by Proposition~\ref{prop:s central} we may
assume that $s$ is not central. Let $\eta \in \cE(\bC^{*F},1)$ and
$\chi\in \cE(\bG^F, s)$ with $\eta = \pi_s^\bG(\chi)$.
We are required to show that $b_{\bG^F}(\chi)$ belongs to
$\bJ_1^\bG (b_{\bC^{*F} }(\eta))$. Let $(\bL_s^*,\la_s)$ be a unipotent
$e$-cuspidal pair of $\bC^{*F}$ with $\eta\in\cE(\bC^{*F},(\bL_s^*,\la_s))$ and
let $(\bL,\la)$ be an $e$-Jordan cuspidal pair below~$(\bG^F,s)$ associated to
$(\bL_s^*,\la_s)$ via Proposition~\ref{prop:En 15} (note that the existence of
$(\bL_s^*,\la_s)$ is guaranteed by generalised $e$-Harish-Chandra theory). By
Proposition~\ref{prop:GeMa 4.7.7-2}, $\chi\in\cE(\bG^F,(\bL,\la))\cup
\cE(\bG^F,(\bL',\la'))$, where $(\bL',\la')$ is an $e$-Jordan cuspidal
pair below~$(\bG^F,s)$ associated to the twin of $(\bL_s^*,\la_s)$ via
Proposition~\ref{prop:En 15}. Consequently, $b_{\bG^F}(\chi)$ is either
$b_{\bG^F}(\bL,\la)$ or $b_{\bG^F}(\bL',\la')$. Since also
$b_{\bC^{*F}}(\eta)=b_{\bC^{*F}}(\bL_s^*,\la_s)$, the claim follows
from Proposition~\ref{prop:several unipotents}.
\end{proof}

\begin{cor}   \label{cor:in Levi}
 Suppose $[\bX,\bX]$ is of exceptional type. Assume $\ell>2$ is good for
 $\bC^*$ and $\ell\ne3$ if $\bC^{*F}$ has a component of type $\tw3D_4$.
 Suppose $\bC^* =\bC_\bb^*$ and that Theorem~{\rm\ref{thm:thmA}} holds for all
 proper $F$-stable Levi subgroups of~$\bG$. Then the conclusion of
 Theorem~$\ref{thm:thmA}$ holds for every $\ell$-element $t\in\bC^{*F}$.
\end{cor}

\begin{proof}
Let $t\in\bC^{*F}$ be an $\ell$-element. The assumptions imply that
$\ell\in\Gamma(\bC,F)$ in the notation of \cite{CE94}. If $t\notin Z({\bC^*})$,
then (the proof of) \cite[Prop.~2.5]{CE94} shows that
$C_{\bG^*}(st)=C_{\bC^*}(t)$ is contained in a proper $e$-split Levi subgroup
$\bM^*$ of~$\bC^*$ and thus also in the proper $e$-split Levi subgroup
$C_{\bG^*}(Z^\circ(\bM^*)_e)$ of $\bG^*$, and hence the conclusion follows
inductively by Proposition~\ref{prop:in Levi}. Thus $t\in Z({\bC^*})$, so
$C_{\bG^*}(st)=C_{\bG^*}(s)$. By Proposition~\ref{prop:isolated reduction}, we
may assume $s$ is isolated in $\bG^*$. By inspection of the lists of
isolated elements in \cite[Prop.~4.9, Tab.~3]{B05}, if $t\in Z({\bC^*})$
is an $\ell$-element then even $t\in Z(\bG^*)$, and we conclude by
Proposition~\ref{prop:t central}. 
\end{proof}

%%%%%%%%%%%%%%%%%%%%%%%%%%%%%%%%%%%%%
\subsection{Theorem~\ref{thm:thmA} for good primes}   \label{subsec:good l}
Suppose that $\ell$ is a good prime for $\bG$. Let
$t\in\bC^{*F}= C_{\bG^*}(s)^F$ be an $\ell$-element and let $\bG(t)\leq\bG$ be
a Levi subgroup of $\bG$ in duality with the Levi subgroup $C_{\bG^*}(t)$ of
$\bG^*$. 

Let $\chi\in\cE(\bG^F,st)$. By properties of Digne--Michel's Jordan
decomposition, $\chi = R_{\bG(t)}^\bG(\hat t\chi_t)$ where
$\chi_t\in\cE(\bG(t)^F,s)$ satisfies
$\pi_{st}^\bG(\chi) = \pi_s^{\bG(t)}(\chi_t)$. By the usual commutation
property of the decomposition map with Lusztig induction (apply
\cite[Prop.~9.6]{CE} with $f= d^{1,\bG}(1))$, $\chi$ lies in the
same $\ell$-block of $\bG^F$ as some constituent of $R_{\bG(t)}^\bG(\chi_t)$.
By results of Cabanes--Enguehard, we have the following. 

\begin{prop}   \label{prop:CE99 2.8}
 With the notation above, every irreducible constituent of
 $R_{\bG(t)}^\bG (\chi_t)$ lies in $b_{\bG^F}(\chi)$. Moreover,
 for every block $c$ of $\bG(t) ^F$ in $\cE _\ell(\bG(t)^F,s)$, there exists a
 unique block $R_{\bG(t)}^\bG(c)$ of $\bG^F$ in $\cE_\ell(\bG^F,s)$ such that
 every irreducible constituent of $R_{\bG(t)}^\bG (\mu)$, for
 $\mu\in\cE(\bG(t)^F,s)\cap\Irr(c)$, lies in $R_{\bG(t)}^\bG(c)$.
\end{prop} 

\begin{proof}
By \cite[Prop.~2.4]{CE99}, $C_{\bG^*}(t)$ is $E_{q,\ell}$-split in $\bG^*$ and
therefore $\bG(t)$ is $E_{q,\ell}$-split in $\bG$. The assertion follows from
\cite[Thms~2.5 and~2.8]{CE99}.
\end{proof} 

\begin{prop} \label{prop:JoEn 2.3.3-2}
 Assume $\ell$ is odd and good for $\bG$ and $\ell>3$ if $\bG^F$ has a
 component of type $\tw3D_4$. Let $(\bL_t^*,\la_t)$ be a unipotent $e$-cuspidal
 pair of $C_{\bG^*}(st)$ of (quasi-)central $\ell$-defect and
 $(\bL,\la)\in J_t^\bG((\bL_t^*,\la_t))$,
 $(\bL(t),\la(t))\in J_t^{\bG(t)}((\bL_t^*,\la_t))$. Suppose that
 $\bGb\leq\bG(t)$. Then
 $$ R_{\bG(t)}^\bG(b_{\bG(t)^F}(\bL(t),\la(t))) = b_{\bG^F}(\bL,\la).$$
\end{prop}

\begin{proof}
Set $b:=b_{\bG^F}(\bL,\la)$, $\bH:= \bG(t)$, $\bL_\bH:= \bL(t)$, and
$\la_\bH:=\la(t)$. Set $\bG_1:=\bG_\ba$ and $\bG_2: = Z^\circ(\bG)\bG_\bb$.
Then $\bG=\bG_1 \bG_2$, the kernel of the multiplication map from
$\tbG:= \bG_1\times\bG_2$ to $\bG$ is a central torus, isomorphic to
$Z(\bG)=Z^\circ(\bG)$, and we may put ourselves
in the setting (and use the notation) of Proposition~\ref{prop:JoEn 2.1.5}.
So, letting $\tilde \bH$ be the inverse image of $\bH$ in $\tbG$, we have
$\tilde\bH= \bH_1\times\bH_2$ with $\bH_i$ an $F$-stable Levi subgroup of
$\bG_i$. Note that by hypothesis $\bH_2=\bG_2$. 
\par
Let $(\bL^*_s,\al)$ be a unipotent $e$-cuspidal pair of $C_{\bG^*}(s)^F$ such
that $\bL^*=C_{\bG^*}(Z^\circ(\bL_s^*)_e)$ is dual to $\bL$ in $\bG$, 
$\pi_s^{\bL}(\la)=\al$ and $(\bL^*_s,\al)\to_{t} (\bL_t^*,\la_t)$.
Similarly, let $(\bL^*_{s,\bH},\al_\bH)$ be a unipotent $e$-cuspidal pair of
$C_{\bH^*}(s)^F$ such that $\bL_\bH^*=C_{\bH^*}(Z^\circ(\bL_{s,\bH}^*)_e)$ is
dual to $\bL_\bH$ in $\bH$, $\pi_s^{\bL_\bH}(\la_\bH)=\al_\bH$ and
$(\bL_{s,\bH}^*,\al_\bH)\to_t (\bL_t^*,\la_t)$ in $C_{\bH^*}(s)$. Note that by 
Proposition~\ref{prop:3D4}, applied with $st$ in place of $s$, $\bL_t^{*F}$
does not have a component of type $\tw3D_4$ and hence $\to_t$ equals $\sim$ in
$C_{\bG^*}(s)$ as well as in $C_{\bH^*}(s)$. Here, we set 
$\bH^*=\bG(t)^*=C_{\bG^*}(t)$. Further, since
$C_{C_{\bH^*}(s)}(t) = C_{\bH^*}(s)$, we may in fact assume that
$(\bL^*_{s,\bH},\al_\bH) = (\bL_t^*,\la_t)$. So, in particular,
$({\bL^*_s}',\al|_{{\bL^*_s}'^F})$ and
$({\bL^*_{s,\bH}}', \al_\bH|_{{\bL^*_{s,\bH}}'^F})$ 
are $C_{\bG^*}(s)^F$-conjugate, where again we denote by $X'$ the derived
subgroup of $X$. By Proposition~\ref{prop:JoEn 2.1.5}(d), it follows that 
$((\bL^*_s)_2',\al_2|_{(\bL^*_s)_2'^F})$ and $((\bL^*_{s,\bH})_2', 
(\al_\bH)_2|_{(\bL^*_{s,\bH})_2'^F})$ are $C_{\bG_2^*}(s_2)^F$-conjugate.
Further, by Proposition~\ref{prop:JoEn 2.1.5}(a), $((\bL^*_s)_2,\al_2)$ and
$((\bL^*_{s,\bH})_2,(\al_\bH)_2)$ are unipotent $e$-cuspidal pairs of
$C_{\bG_2^*}(s_2) = C_{\bH_2^*}(s_2)$. Therefore, by the argument at the
beginning of the proof of Proposition~\ref{prop:En 17} and up to conjugation in 
$C_{\bG_2^*}(s_2)^F$, we may assume that
$$((\bL^*_s)_2,\al_2) = ((\bL^*_{s,\bH})_2,(\al_\bH)_2),$$
and hence by Proposition~\ref{prop:JoEn 2.1.5}(b) that 
$$(\bL_2,\la_2) = ((\bL_\bH)_2,(\la_\bH)_2).$$
\par
Now let $\chi_\bH\in\cE(\bH^F,(\bL_\bH,\la_\bH))$ and let $\chi$ be a
constituent of $R_\bH^\bG(\chi_\bH)$. By Proposition~\ref{prop:JoEn 2.1.5}(c),
$\tchi_\bH= (\chi_\bH)_1\otimes(\chi_\bH)_2$ for some
$(\chi_\bH)_i\in \cE(\bH_i^F,((\bL_H)_i,(\la_\bH)_i))$, $i=1,2$, and similarly
$\tchi = \chi_1\otimes \chi_2$ with $\chi_i\in\cE(\bG^F,(\bH_i,(\chi_\bH)_i))$.
Since $\bH_2= \bG_2$, we have $\chi_2=(\chi_\bH)_2$ and as observed above,
$(\bL_2,\la_2) = ((\bL_\bH)_2,(\la_\bH)_2)$. Hence,
$\chi_2\in\cE(\bG_2^F,(\bL_2,\la_2))\subseteq\cE(\bG_2^F,s_2)\cap\Irr(b_2)$.
On the other hand, $\chi_1\in \cE(\bG_1^F,s_1)$ and by Lemma~\ref{lem:Ga},
$\cE(\bG_1^F,s_1) = \cE(\bG_1^F, (\bL_1,\la_1))$. It follows by
Proposition~\ref{prop:JoEn 2.1.5}(c) that
$\chi\in \cE(\bG^F,(\bL,\la))\subseteq\Irr(b)$. Since
$\chi_\bH\in\Irr(b_{\bH^F}(\bH,\la_\bH)$,
the result follows by Proposition~\ref{prop:CE99 2.8}.
\end{proof} 

\begin{prop}   \label{prop:good}
 Suppose that $\ell$ is odd and good for $\bG$, and $\ell >3$ if $\bG^F$ has a
 component of type $\tw3D_4$. Suppose also that Theorem~$\ref{thm:thmA}$ holds
 for all proper $F$-stable Levi subgroups of $\bG$. Then
 Theorem~$\ref{thm:thmA}$ holds for $\bG$ and $\ell$.
\end{prop}

\begin{proof}
By hypothesis and Proposition~\ref{prop:in Levi}, we may assume $\bG(t)$ is not
contained in any proper $e$-split Levi subgroup of $\bG$. Thus by
\cite[Prop.~2.5]{CE94} we may assume that $\bGb\leq\bG(t)$. 
By Propositions~\ref{prop:CE99 2.8} and ~\ref{prop:JoEn 2.3.3-2} and the
paragraph preceding Proposition~\ref{prop:CE99 2.8} it suffices to
prove the result for the case $\bG=\bG(t)$ and $t=1$. This case was settled in
Proposition~\ref{prop:t central}.
\end{proof}

%%%%%%%%%%%%%%%%%%%%%%%%%%%%%%%%%%%%%
\subsection{Proof of Theorem~\ref{thm:thmA}}   \label{subsec:cases}
We are now ready to complete the proof of our main theorem by checking
individually the various isolated blocks at bad primes in simple groups of
exceptional type.

\begin{proof}[Proof of Theorem~\ref{thm:thmA}]
Let $\bG$ be as in the statement and assume that Theorem~\ref{thm:thmA} holds
for every proper $F$-stable Levi subgroup of $\bG$. If $\ell$ is odd and good
for $\bG$ and $\ell> 3$ if $\bG^F$ has a component of type $\tw3D_4$, the
result follows by Proposition~\ref{prop:good}. If $\bG$ does not have any
factors of exceptional type and $\ell=2$, then $\cE_\ell(\bG^F,s)$ is a single
$\ell$-block (see \cite[Thm~21.14]{CE}) and the claim holds trivially.
So we are reduced to the case that $\bG^F$ has a factor of exceptional type and
that $\ell$ is bad for $\bG$, or $\ell=3$ and $\bG^F$ has a component of type
$\tw3D_4$. By Proposition~\ref{prop:isolated reduction} we may assume that $s$
is isolated in $\bG^*$ and by Proposition~\ref{prop:s central} that $s$ is not
central in $\bG^*$.
\par
Suppose that $\ell=3$ and $\bG^F$ has a component of type $\tw3D_4$. By
Proposition~\ref{prop:3D4} this implies $[\bX,\bX]$ is of exceptional type.
Then, $\bG$ is of type $D_4$, $D_4A_1$ or $D_4A_2$. Since groups of type~$A$
have no non-central isolated elements, the $D_4$-component of $s$ in $\bG^*$ is
non-central and isolated. Thus, by \cite[Tab.~9]{KM}, all components of
$C_{\bG^*}(s)$ are of type $A$ and unless $\bG$ has type $D_4A_2$ they are all
of type $A_1$. By Corollary~\ref{cor:in Levi}, we may inductively assume that
$\bG$ is of type $D_4A_2$ and hence again by \cite[Tab.~9]{KM} the rational
type of $C_{\bG^*}(s)$ is $A_1(q^3)A_1(q).A_2(\eps q)$ for some
$\eps\in\{\pm1\}$. If $3{\not|}(q-\eps)$ then we conclude by
Corollary~\ref{cor:in Levi}, while in the opposite case, by
\cite[Prop.~3.3]{CE}, $C_{\bG^*}(s)$ has only one conjugacy class of unipotent
$e$-cuspidal pairs. It follows that $\cE_3(\bG^*,s)$ is a single 3-block and
the claim holds trivially.
\par
Thus, we may assume that $\ell$ is bad for $\bG$, $s$ is isolated and
non-central, and either $[\bG,\bG]$ is simple of exceptional type or of type
$E_6A_1$.
\par
We now discuss these remaining possibilities according to the structure of
$[\bG,\bG]^F$.
\medskip

\noindent
{\bf Groups $G_2(q)$} 
\par
Here the centralisers of isolated elements $s\ne1$ are of rational types
$A_1(q)^2$, $A_2(q)$ and $\tw2A_2(q)$, and in each case, $\cE_\ell(\bG^F,s)$ is
a single $\ell$-block for the appropriate primes $\ell$, see \cite[Tab.~9]{KM}.
Thus Theorem~\ref{thm:thmA} is trivially satisfied.
\medskip

\noindent
{\bf Groups $F_4(q)$}\par
For $\ell=3$ let $1\ne s\in\bG^{*F}$ be an isolated $2$-element. Then
$\bC^*=C_{\bG^*}(s)$ has only factors of classical type, so~3 is a good prime
for $\bC^*$, and moreover the assumptions of Corollary~\ref{cor:in Levi} are
satisfied, whence we are done.
For $\ell=2$ there is nothing to prove as by \cite[Tab.~2]{KM} for all
isolated 3-elements~$s\ne1$, $\cE_2(G,s)$ is a single 2-block.
\medskip

\noindent
{\bf Groups $E_6(q)$ and $\tw2E_6(q)$}\par
We give the arguments for $E_6(q)$, the case of $\tw2E_6(q)$ being entirely
similar.
For $\ell=3$ the centralisers of isolated 2-elements $1\ne s\in\bG^*$ are of
rational type $A_5(q)A_1(q)$. We discuss the various possibilities for
$3$-elements
$t\in \bC^{*F}$. If $t$ is central, so $\ell=3$ divides $|Z(\bG^*)^F|$ and
hence $e=1$, we conclude by Proposition~\ref{prop:t central}.
Otherwise, by inspection $C_{\bG^*}(st)$ lies in a proper $e$-split Levi
subgroup of $\bG^*$ and we can apply Proposition~\ref{prop:in Levi}. For
$\ell=2$ again by Table~\ref{tab:quasi-E6} for all isolated 3-elements~$s\ne1$,
$\cE_2(G,s)$ is a single 2-block.
\medskip

\noindent
{\bf Groups $E_6(q)A_1(q)$ and $\tw2E_6(q)A_1(q)$}\par
Here we can argue in a completely similar fashion as in the previous case since
the $A_1$-factor has no non-central isolated elements.
\medskip

\noindent
{\bf Groups $E_7(q)$}\par
For $\ell=3$ by \cite[Tab.~4]{KM} we only need to consider the blocks in
Table~\ref{tab:quasi-E7} below.
Here all centralisers of isolated 2-elements satisfy the assumptions of
Corollary~\ref{cor:in Levi} and we are done. Note that $E_7(2)$ does not need
to be considered as it has no (isolated) semisimple 2-elements.
For $\ell=2$ again $\cE_2(G,s)$ is a single 2-block for all isolated
3-elements~$s\ne1$ by \cite[Tab.~4]{KM}.
\medskip

\noindent
{\bf Groups $E_8(q)$}\par
First assume that $\ell=5$. Let $1\ne s\in\bG^{*F}$ be an isolated $5'$-element
and $\bC^*=C_{\bG^*}(s)$. Then~5 is good for $\bC^*$ and the assumptions of
Corollary~\ref{cor:in Levi} are satisfied for $\bC^*$ (see \cite[Tab.~7
and~8]{KM} and Tables~\ref{tab:q=1(5)} and~\ref{tab:q=2(5)} below). This
completes the argument when $\ell=5$.
\par

Now assume that $\ell=3$ and let $e=e_3(q)$. Let $1\ne s\in\bG^{*F}$ be an
isolated $3'$-element and $\bC^*=C_{\bG^*}(s)$. If $\bC^*$ has only classical
factors not of rational type $\tw3D_4$, then~3 is a good prime for $\bC^*$ and
we can apply exactly the same
argument as in the case $\ell=5$ to conclude. The only centraliser for which
this approach fails is when $s$ is an involution with $\bC^*$ of rational
type $E_7(q)A_1(q)$.

The Harish-Chandra series in this case are listed in Table~\ref{tab:quasi-E8}
below (copied from \cite[Tab.~6]{KM}) for $e=1$; for $e=2$ we have the Ennola
dual situation which can be treated in exactly the same manner. We discuss the
various $3$-elements $t\in\bC^{*F}$. Assume that $t$ has a non-central
component in the $A_1$-factor. Then $C_{\bG^*}(st)$ is contained in a 1-split
Levi subgroup of $\bG^*$ of rational type $E_7(q).\Ph1$, and we may conclude by
Proposition~\ref{prop:in Levi}. Thus, the centraliser of~$t$ does not contain
the whole $E_7$-factor. But now by inspection the centraliser of any element of
order~3 in $E_7(q)$ is either of type $A_5A_2$, or it is contained in a proper
1-split Levi subgroup of $E_7(q)$. Thus, any non-trivial 3-element of $E_7(q)$
none of whose powers has centraliser $A_5A_2$ has its centraliser in a proper
1-split Levi subgroup and Proposition~\ref{prop:in Levi} applies.
\par
So finally assume $t$ is such that $C_{\bG^*}(st^k)^F=A_5(q)A_2(q)A_1(q)$ for
some~$k\ge1$. It can be checked using \Chevie{} and the known block
distribution for $\cE(\bG^F,s)$ from \cite[Prop.~6.7]{KM} that all but four
classes of maximal tori $\bT^*$ of $A_5(q)A_2(q)A_1(q)$ have the property that
all constituents of $\RTG(\hat s)$ lie in the same 3-block of $\bG^F$, namely
the semisimple block in $\cE_3(\bG^F,s)$. Now let $\bT^*$ be the maximally
split torus of $C_{\bG^*}(st)$; note that all factors of this centraliser have
untwisted type~$A$. So $R_{\bT^*}^{C_{\bG^*}(st)}(1)$ contains all unipotent
characters of $C_{\bG^*}(st)$. Note that the maximally split torus is uniquely
determined up to conjugacy inside $C_{\bG^*}(st)$ by its order. Thus, if
$\bT^*$ is not one of the four exceptions mentioned before, then by
Lemma~\ref{lem:RTG} all $\chi\in\cE(\bG^F,st)$ lie in the semisimple block in
$\cE_3(\bG^F,s)$, and we may conclude with Proposition~\ref{prop: ss} that
Theorem~\ref{thm:thmA} holds in this case. The excluded maximal tori $\bT^*$,
intersected with the $A_5(q)A_2(q)$-factor of $C_{\bG^*}(st^k)^F$,
have orders $\Phi_3^3\Phi_1$ or $\Phi_2\Phi_3^2\Phi_6$, and they are maximally
split in a centraliser in $A_5(q)A_2(q)$ only when this centraliser is
contained in a subgroup $\Phi_3^3A_1(q)$, respectively in $\bT^*$. But the
second does not occur as the centraliser of a 3-element in $A_5(q)A_2(q)$,
while for elements with centraliser the first type, no power has centraliser
$A_5(q)A_2(q)A_1(q)$.
\par
For $\ell=2$, by \cite[Tab.~5]{KM} there are only two types of non-central
isolated elements $s\in\bG^{*F}$ of order~3 to consider with centralisers as
listed in Table~\ref{tab:quasi-E8l=2}, which is taken from loc.~cit., for
$q\equiv1\pmod4$. The case $q\equiv3\pmod4$ is Ennola dual to this one and
analogous arguments apply to it.

\begin{table}[htbp]
\caption{Harish-Chandra series in some isolated 2-blocks of $E_8(q)$, $q\equiv1\pmod4$}   \label{tab:quasi-E8l=2}
$$\begin{array}{|c|r|r|l|ll|}
\hline
 \text{No.}& C_{\bG^*}(s)^F& \bL^F& C_{\bL^*}(s)^F& \la& W_{\bG^F}(\bL,\la)\\
\hline
 3& E_6(q).A_2(q)& \emptyset& \bL^{*F}& 1& E_6\ti A_2\\
  &         & D_4& \bL^{*F}& D_4[1]& G_2\ti A_2\\
 4&         & E_6& \bL^{*F}& E_6[\theta^{\pm1}]& A_2\\
\hline
 5& \tw2E_6(q).\tw2A_2(q)& A_1^3& \Ph1^5\Ph2^3& 1& F_4\ti A_1\\
  &         & D_4& \Ph1^4\Ph2^2.\tw2A_2(q)& \phi_{21}& F_4\\
  &         & D_6& \Ph1^2\Ph2.\tw2A_5(q)& \phi_{321}& A_1\ti A_1\\
  &         & E_7& \Ph1.\tw2A_5(q)\tw2A_2(q)& \phi_{321}\otimes\phi_{21}& A_1\\
  &         & E_7& \Ph1\Ph2.\tw2E_6(q)& \tw2E_6[1]& A_1\\
  &         & E_8& C_{\bG^*}(s)^F& \tw2E_6[1]\otimes\phi_{21}& 1\\
 6&         & E_7& \Ph1\Ph2.\tw2E_6(q)& \tw2E_6[\theta^{\pm1}]& A_1\\
  &         & E_8& C_{\bG^*}(s)^F& \tw2E_6[\theta^{\pm1}]\otimes\phi_{21}& 1\\
\hline
\end{array}$$
\end{table}

The involution centralisers in $C_{\bG^*}(s)^F=E_6(q).A_2(q)$ either
lie in a proper 1-split Levi subgroup, or equal $A_5(q)A_2(q)A_1(q)$. Thus,
again, we only need to worry about $2$-elements $t\ne1$ such that
$C_{\bG^*}(st^k)=A_5(q)A_2(q)A_1(q)$ for some $k\ge1$. In this case all
constituents of the Deligne--Lusztig characters for maximal tori of
$C_{\bG^*}(st^k)$ lie in the semisimple $2$-block in $\cE_2(\bG^F,s)$ as
described in \cite[Prop.~6.4]{KM} and we conclude with Lemma~\ref{lem:RTG}.
\par
Finally assume $C_{\bG^*}(s)^F=\tw2E_6(q).\tw2A_2(q)$. Here, centralisers of
involutions are either contained in a proper 1-split Levi subgroup and we
are done by induction, or are of rational type $\tw2A_5(q)A_1(q).\tw2A_2(q)$ or
$\tw2D_5(q)\Phi_2.\tw2A_2(q)$. By a \Chevie-computation, all maximal tori
$\bT^*$ of either of the latter two subgroups have the property that all
constituents of $\RTG(\hat s)$ lie in the semisimple block of $\cE_2(\bG^F,s)$
as described in \cite[Prop.~6.4]{KM}, and so we may conclude as before.
\end{proof}

\begin{proof}[Proof of Corollary~$\ref{cor:cor1}$]
By \cite[Thm~A, Thm~A.bis]{En00}, the Jordan correspondents of $\chi$ and
$\chi'$ lie in the same unipotent $\ell$-block of $C_{\bG^*}(r)^F$. Now the
result is immediate from Theorem~\ref{thm:thmA}.
\end{proof}

%%%%%%%%%%%%%%%%%%%%%%%%%%%%%%%%%%%%%%%%%%%%%%%%%%%%%%%%%%%%%%%%%%%%%%%%%
\section{Descent to quasi-simple groups of types $E_6$ and $E_7$}   \label{sec:sc type}

Let $\bX$ be as in Theorem~\ref{thm:thmA} such that $\bG:=[\bX,\bX]$ is
simple of simply connected type $E_6$ or $E_7$; so
$\bG\hookrightarrow\bX$ is a regular embedding. Now $\bG_\ad:=\bX/Z(\bX)$ is
simple of adjoint type, and $\bG_\ad^F=(\bX/Z(\bX))^F\cong \bX^F/Z(\bX^F)$ since
$Z(\bX)$ is connected, so Theorem~\ref{thm:thmA} immediately gives a
description of the $\ell$-blocks of $\bG_\ad^F$, by only considering those
characters that are trivial on $Z(\bX^F)$.   \par
We also obtain strong information on the $\ell$-blocks of groups of simply
connected type. Namely, with $m:=|Z(\bG^F)|$ we have: if $m=1$ then
$\bG^F=(\bG/Z(\bG))^F=\bG_\ad^F$ which was discussed above,
while if $m\ne1$ then $\bG^F$ is quasi-simple with $m=3$ for type $E_6$ and
$m=2$ for type~$E_7$. For general facts on covering blocks see \cite[Sec.~6.8]{Li18}.

\begin{prop}   \label{prop:sc}
 In the setting introduced above, let $\ell$ be a non-defining prime for $\bG$,
 let $b=b_{\bG^F}(\bL,\la)$ be an $\ell$-block of $\bG^F$ in $\cE_\ell(\bG^F,s)$
 for $s\in\bG^{*F}$ a semisimple $\ell'$-element,  and let $B$ be an
 $\ell$-block of $\bX^F$ covering~$b$.
 \begin{enumerate}[\rm(a)]
  \item If $b$ is $\bX^F$-invariant, the members of $\Irr(b)$ are the
   constituents of the restriction to $\bG^F$ of the members of $\Irr(B)$. If
   moreover $C_{\bG^*}(s)^F=C_{\bG^*}^\circ(s)^F$ and $\ell\ne m$, restriction
   defines a height preserving bijection from $\Irr(B)$ to $\Irr(b)$.
  \item If $b$ is not $\bX^F$-invariant, then
   $C_{\bG^*}^\circ(s)^F<C_{\bG^*}(s)^F$, in particular $C_{\bG^*}(s)$ is
   disconnected, $\ell\ne m$, the block $b$ is $\bX^F$-conjugate to $m$
   distinct blocks of $\bG^F$, all $\tchi\in\Irr(B)$ have reducible restriction
   to $\bG^F$ and the restriction of every member of $\Irr(B)$ to $\bG^F$
   contains one constituent in each of these blocks. This defines a height
   preserving bijection from $\Irr(B)$ to $\Irr(b)$.
 \end{enumerate}
 In particular, the $\ell$-block distribution of $\bG^F$ is determined, up to
 the labelling of characters in $\bX^F$-orbits in case~{\rm(b)}, by
 Theorem~$\ref{thm:thmA}$.
\end{prop}

\begin{proof}
The first claim in~(a) is a standard fact about covering blocks. Assume
$C_{\bG^*}(s)^F=C_{\bG^*}^\circ(s)^F$ and $m\ne\ell$. We claim that
$C_{\bG^*}(st)^F=C_{\bG^*}^\circ(st)^F$ for all $\ell$-elements
$t\in C_{\bG^*}(s)^F$.
If not, then $|C_{\bG^*}(st)^F:C_{\bG^*}^\circ(st)^F|=m>1$. But
$C_{\bG^*}(st)=C_{C_{\bG^*}(s)}(t)$, and $m$ is prime to ~$\ell$, hence to the
order of~$t$, contradicting \cite[Prop.~14.20]{MT}. Thus all characters in
$\Irr(B)$ restrict irreducibly to $\bG^F$. This gives the last claim in~(a).
\par
So now assume that $b$ is not $\bX^F$-invariant. Then no $\chi\in\Irr(b)$ is
$\bX^F$-invariant, so necessarily $C_{\bG^*}(s)^F>C_{\bG^*}^\circ(s)^F$ (and
hence $C_{\bG^*}(s)$ is disconnected), which again by \cite[Prop.~14.20]{MT}
implies that $o(s)$ is divisible by $m$ and so $\ell\ne m$. Further, $B$ covers
the $m=|\bX^F:\bG^FZ(\bX^F)|$ distinct $\bX^F$-conjugates of $b$ and the
restriction to $\bG^F$ of every character $\tchi\in\Irr(B)$ has
constituents in each of them. Since $\bX^F/\bG^F$ is cyclic, all restrictions
are multiplicity-free and thus have $m$ constituents. The last claim of~(b)
then follows.
\end{proof}

\begin{rem}
(1) If $\bG^F$ is of adjoint type then $m=1$ and we are in the situation of
case~(a).\par
(2) If $\ell\ne2$ then case~(a) occurs precisely when the $\bG^F$-class of
$(\bL,\la)$ is $\bX^F$-invariant, by \cite[Thm~A]{KM15} (but see the
counter-example in \cite[Exmp.~3.16]{KM15} when $\ell=2$.)\par
(3) In (a), if either $C_{\bG^*}^\circ(s)^F<C_{\bG^*}(s)^F$ or $\ell=m$, there
will in general not exist a height preserving bijection.   \par
(4) The proof shows that in case~(b) the centralisers $C_{\bG^*}(st)$ are
disconnected for all $\ell$-elements $t\in C_{\bG^*}(s)$.
%; indeed, the
%outer automorphism induced by $C_{\bG^*}(s)$ on $C_{\bG^*}^\circ(s)$ centralises
%a Sylow $\ell$-subgroup of $C_{\bG^*}^\circ(s)$.
\end{rem}

\begin{exmp}
 In case~(a) of Proposition~\ref{prop:sc}, $C_{\bG^*}(s)$ could be connected or
 disconnected: an example for the first case is the block~13 in
 Table~\ref{tab:quasi-E6} below, an example for the second is the block~1 in
 Table~\ref{tab:quasi-E6}. Even if $C_{\bG^*}(s)$ is connected but $\ell=m$
 there may exist $\ell$-elements $t\in C_{\bG^*}(s)^F$ with $C_{\bG^*}(st)$
 disconnected (e.g., taking $s=1$).
\end{exmp}

%%%%%%%%%%%%%%%%%%%%%%%%%%%%%%%%%%%%%%%%%%%%%%%%%%%%%%%%%%%%%%%%%%%%%%%%%
\section{On $e$-Harish-Chandra series in exceptional groups}   \label{sec:exc type}
In this section we complement our results from \cite{KM} in several ways. First
we parametrise the isolated blocks for bad primes in exceptional groups of type
$E_6$ and $E_7$ with connected centre and verify the validity of an
$e$-Harish-Chandra theory in this situation. Second, we give the
block distribution for some isolated 5-blocks in groups of type $E_8$
inadvertently omitted in \cite{KM}. Finally, we give the proofs of the
extensions of results from \cite{KM} and \cite{KM15} to the situation
considered in the present paper. The $e$-cuspidal pairs in the subsequent
tables can be determined, for example, as explained in \cite{Tay22}.

%%%%%%%%%%%%%%%%%%%%%%%%%%%%%%%%%%%%%
\subsection{Isolated blocks in exceptional groups of adjoint type} \label{subsec:adjoint}

In this section $\bG$ is a group with connected centre such that $[\bG,\bG]$ is
simple of simply connected type $E_6$ or $E_7$ and $F:\bG\rightarrow\bG$ is a
Frobenius endomorphism with respect to an $\FF_q$-structure.
\par
It follows by looking at the root datum that $\bG$ is isomorphic to its dual
$\bG^*$. Since $\bG$ has connected centre, all centralisers of semisimple
elements in $\bG^*$ are connected and thus the notions of isolated and
quasi-isolated elements do coincide. If $s\in\bG^*$ is isolated,
then so is $sz$ for any $z\in Z(\bG^*)$, with the same centraliser, and this
defines an equivalence relation on the set of conjugacy classes of isolated
semisimple elements. We describe the $\ell$-block subdivision of $\cE(\bG,s)$
for $s\in\bG^{*F}$ an isolated $\ell'$-element and $\ell\in\{2,3\}$ not
divding~$q$. Note that
the blocks corresponding to two isolated elements $s$ and $sz$, for
$z\in Z(\bG^{*F})$, are obtained from one another by tensoring with a linear
character $\hat z$ of $\bG^F$ (see \cite[Prop.~2.5.21]{GM20}). Also note that
the case of $s=1$, that is, of unipotent blocks has been dealt with by Enguehard
\cite{En00}, so we may assume $s$ is non-central. As before, let $e=e_\ell(q)$
be the order of $q$ modulo $\ell$ when $\ell$ is odd, respectively the order of
$q$ modulo~4 when $\ell=2$.

We first determine the decomposition of Lusztig induction for
$e$-Harish-Chandra series of $\bG$ in $\cE(\bG^F,s)$.

\begin{prop}   \label{prop:E6 RLG}
 Let $\bG$ be as above, $\ell\in\{2,3\}$ not dividing~$q$, $s\in\bG^{*F}$ a
 non-central isolated $\ell'$-element, and set $e=e_\ell(q)$. Then we have:
 \begin{enumerate}[\rm(a)]
  \item If $[\bG,\bG]^F=E_6(q)_\SC$ then $\cE(\bG^F,s)$ is the disjoint union
   of the $e$-Harish-Chandra series listed in Table~$\ref{tab:quasi-E6}$.
  \item If $[\bG,\bG]^F=\tw2E_6(q)_\SC$ then the $e$-Harish-Chandra series in
   $\cE(\bG^F,s)$ are the Ennola duals of those in Table~$\ref{tab:quasi-E6}$.
  \item If $[\bG,\bG]^F=E_7(q)_\SC$ then $\cE(\bG^F,s)$ is the disjoint union
   of the $e$-Harish-Chandra series listed in Table~$\ref{tab:quasi-E7}$ when
   $e=1$, respectively their Ennola duals when $e=2$.
  \item The assertion of \cite[Thm~1.4]{KM} continues to hold for $\bG$.
 \end{enumerate}
\end{prop}

In Tables~\ref{tab:quasi-E6} and~\ref{tab:quasi-E7}, we give the data
relative to $\bbG:=\bG/Z(\bG)$, a simple group of adjoint type. The numbering
of blocks follows \cite[Tab.~3 and~4]{KM}.

\begin{table}[htbp]
\caption{$e$-Harish-Chandra series in $\bbG^F=E_6(q)_\ad$}   \label{tab:quasi-E6}
\[\begin{array}{|r|r|l|llll|}
\hline
 \text{No.}& C_{\bbG^*}(s)^F& (\ell,e)& \bbL^F& C_{\bbL^*}(s)^F& \la& W_{\bbG^F}(\bbL,\la)\\
\hline\hline
 1&        A_2(q)^3& (2,1)& \Ph1^6& \bbL^{*F}& 1& A_2^3\\
\hline
 2&        A_2(q^3)& (2,1)& \Ph1^2.A_2(q)^2& \Ph1^2\Ph3^2& 1& A_2\\
\hline
 6& A_2(q^2).\tw2A_2(q)& (2,1)& \Ph1^3.A_1(q)^3& \Ph1^3\Ph2^3& 1& A_2\ti A_1\\
  &                    &      &  \Ph1^2.D_4(q)& \Ph1^2\Ph2^2.\tw2A_2(q)& \phi_{21} & A_2\\
\hline\hline
 7&        A_2(q)^3& (2,2)& \Ph1^2\Ph2^3.A_1(q)& \Ph1^3\Ph2^3& 1& A_1^3\\
  &                & & \Ph1\Ph2^2.A_3(q)& \Ph1^2\Ph2^2.A_2(q)& \phi_{21}\ (3\ti)& A_1\ti A_1\\
  &                & & \Ph2.A_5(q)& \Ph1\Ph2.A_2(q)^2& \phi_{21}^{\otimes2}\ (3\ti)& A_1\\
  &                & & \bbG^F& C_{\bbG^*}(s)^F& \phi_{21}^{\otimes3}& 1\\
\hline
 8&        A_2(q^3)& (2,2)& \Ph2.A_2(q^2)A_1(q)& \Ph1\Ph2\Ph3\Ph6& 1& A_1\\
  &                & & \bbG^F& C_{\bbG^*}(s)^F& \phi_{21}& 1\\
\hline
 12&  A_2(q^2).\tw2A_2(q)& (2,2)& \Ph1^2\Phi_2^4& \bbL^{*F}& 1& A_2\ti A_2\\
\hline\hline
 13&      A_5(q)A_1(q)& (3,1)& \Ph1^6& \bbL^{*F}& 1& A_5\ti A_1\\
\hline\hline
 14&      A_5(q)A_1(q)& (3,2)& \Ph1^2\Ph2^4& \bbL^{*F}& 1& C_3\ti A_1\\
%%%  &                  &      & \Ph1^2\Ph2^2.A_2(q)& q\Ph2& & ?\\
\hline
 15&                  &      & \Ph2.A_5(q)& \bbL^{*F}& \phi_{321}& A_1\\
\hline
\end{array}\]
\end{table}

\begin{table}[htbp]
\caption{$e$-Harish-Chandra series in $\bbG^F=E_7(q)_\ad$}   \label{tab:quasi-E7}
\[\begin{array}{|r|r|l|llll|}
\hline
 \text{No.}& C_{\bbG^*}(s)^F& (\ell,e)& \bbL^F& C_{\bbL^*}(s)^F& \la& W_{\bbG^F}(\bbL,\la)\\
\hline\hline
 1&         A_5(q)A_2(q)& (2,1)& \Ph1^7& \bbL^{*F}& 1& A_5\ti A_2\\
\hline
 2& \tw2A_5(q)\tw2A_2(q)& (2,1)& \Ph1^4.(A_1(q)^3)'& \Ph1^4\Ph2^3& 1& C_3\ti A_1\\
  &                     &      & \Ph1^3.D_4(q)& \Ph1^3\Ph2^2.\tw2A_2(q)& \phi_{21}& C_3\\
  &                     &      & \Ph1.D_6(q)& \Ph1\Ph2.\tw2A_5(q)& \phi_{321}& A_1\\
  &                     &      & E_7(q)& C_{\bbG^*}(s)^F& \phi_{321}\otimes\phi_{21}& 1\\
\hline\hline
 3&  D_6(q)A_1(q)& (3,1)& \Ph1^7& \bbL^{*F}& 1& D_6\ti A_1\\
\hline
 4&              &      & \Ph1^3.D_4(q)& \bbL^{*F}& D_4[1]& B_2\ti A_1\\
\hline
 5&      A_7(q)& (3,1)& \Ph1^7& \bbL^{*F}& 1& A_7\\
\hline
 6&  \tw2A_7(q)& (3,1)& \Ph1^4.(A_1(q)^3)'& \Ph1^4\Ph2^3& 1& C_4\\
\hline
 7&              &      & \Ph1.D_6(q)& \Ph1\Ph2.\tw2A_5(q)& \phi_{321}& A_1\\
\hline
12& A_3(q)^2A_1(q)& (3,1)& \Ph1^7& \bbL^{*F}& 1& A_3^2\ti A_1\\
\hline
13& \tw2A_3(q)^2A_1(q)& (3,1)& \Ph1^5.A_1(q)^2& \Ph1^5\Ph2^2& 1& B_2^2\ti A_1\\
\hline
14& A_3(q^2)A_1(q)& (3,1)& \Ph1^4.(A_1(q)^3)'& \Ph1^4\Ph2^3& 1& A_3\ti A_1\\
%\hline\hline
% 10b&                      & (3,3)& \Ph3^2.A_1(q^3)& \Ph2\Ph3^2\Ph6& 1& G_5\\
\hline
\end{array}\]
\end{table}

\begin{proof}
Let $\bG_0:=[\bG,\bG]$, a simple group of simply connected type, and consider
the regular embedding $\bG_0\hookrightarrow\bG$. Then for any $F$-stable Levi
subgroup $\bL_0\le\bG_0$ we have
$$\Ind_{\bG_0^F}^{\bG^F}\circ R_{\bL_0}^{\bG_0}
   =\RLG\circ\Ind_{\bL_0^F}^{\bL^F}$$
(see \cite[Prop.~3.2.9]{GM20}), where $\bL= Z(\bG)\bL_0$ is the corresponding
Levi subgroup of $\bG$. Now above every character
in $\cE(\bG_0^F,s)$ there are $|C_{\bG^*}(s)^F:C_{\bG_0^*}(s)^F|$
characters of $\bG^F$, lying in distinct Lusztig series, and similarly, above
every character in $\cE(\bL_0^F,s)$ there are
$|C_{\bL^*}(s)^F:C_{\bL_0^*}(s)^F|$ characters of $\bL^F$, lying in
distinct Lusztig series. Write
$\Ind_{\bL_0^F}^{\bL^F}(\la)=\sum_i\la_i$ with $\la_i\in\Irr(\bL^F)$. Thus all
constituents $\la_i$ lie in distinct Lusztig series of $\bL^F$, so any
summand $\RLG(\la_i)$ of
$(\RLG\circ\Ind_{\bL_0^F}^{\bL^F})(\la)=\sum_i\RLG(\la_i)$ lies inside
a fixed Lusztig series of $\bG^F$ and those lying inside a fixed Lusztig series
are equal. In particular, knowing $R_{\bL_0}^{\bG_0}$ and the decomposition of
$\Ind_{\bG_0^F}^{\bG^F}$ along Lusztig series, we can recover the $\RLG(\la_i)$
uniquely.   \par
Thus we may recover the
stated decomposition of $\RLG(\la)$ from the one for $\bG_0^F$ that was
determined in \cite[Prop.~4.1 and~5.1]{KM} (with the correction given in the
proof of \cite[Thm~3.14]{KM15}, see Remark~\ref{rem:3.14} below), from where our
numbering of cases is taken. This also shows that Harish-Chandra series for
the case when $\bG_0^F=\tw2E_6(q)_\SC$ are Ennola dual to those for
the untwisted case, and (as in \cite{KM}) we do not print them here.
The existence of an $e$-Harish-Chandra theory as in \cite[Thm~1.4]{KM} follows
(see \cite[Def.~2.9]{KM}).
\end{proof}

\begin{lem}   \label{lem:E6}
 Let $\bG$, $\bL$ and $\ell$ be as in Proposition~$\ref{prop:E6 RLG}$, with
 $e=e_\ell(q)$. Then:
 \begin{enumerate}[\rm(a)]
  \item $\bL=C_\bG(Z(\bL)_\ell^F)$ and $\bL$ is $(e,\ell)$-adapted; and
  \item in Tables~$\ref{tab:quasi-E6}$ and~$\ref{tab:quasi-E7}$, $\la$ is of
   quasi-central $\ell$-defect precisely in the numbered lines.
 \end{enumerate}
\end{lem}

In fact, in all numbered lines except~6, 7 and~8 in Table~\ref{tab:quasi-E6},
$\la$ is even of central $\ell$-defect. We had shown in
\cite[Lemma~2.7(b)]{KM15} that regular embeddings do preserve the property of
having quasi-central $\ell$-defect.

By Proposition~\ref{prop:E6 RLG} and Lemma~\ref{lem:E6} the assumptions of
\cite[Prop.~2.17]{KM} are satisfied, so each $e$-Harish-Chandra series in the
two tables is contained in a unique $\ell$-block of $\bG^F$.

\begin{prop}   \label{prop:E6E7-defgrp}
 Let $[\bG,\bG]^F=E_6(q)_\SC$ (resp.~$[\bG,\bG]^F=E_7(q)_\SC$). Then for any
 isolated non-central $\ell'$-element $s\in\bG^{*F}$ the $\ell$-block
 subdivision of $\cE(\bG^F,s)$ is as indicated by the horizontal lines in
 Tables~$\ref{tab:quasi-E6}$ and~$\ref{tab:quasi-E7}$.

 For each $\ell$-block corresponding to one of the cases in the tables there is
 a defect group $P\leq N_{\bG^F}(\bL,\la)$ with the structure described in
 \cite[Thm~1.2]{KM}.

 The analogous, Ennola dual statement holds for $[\bG,\bG]^F=\tw2E_6(q)_\SC$. 
\end{prop}

\begin{proof}
This is entirely analogous to the proofs of \cite[Prop.~4.3 and~5.3]{KM}.
\end{proof}

\begin{rem}
The group $\bG/Z(\bG)$ is simple of adjoint type, and as $Z(\bG)$ is connected,
$(\bG/Z(\bG))^F=\bG^F/Z(\bG)^F$, so the above result provides also a
parametrisation of the isolated $\ell$-blocks of the groups $E_6(q)_\ad$,
$\tw2E_6(q)_\ad$ and $E_7(q)_\ad$ for $\ell=2,3$.
\end{rem}

%%%%%%%%%%%%%%%%%%%%%%%%%%%%%%%%%%%%%
\subsection{Some 5-blocks in $E_8(q)$} \label{subsec:E8, l=5}
We deal with a situation missed in our paper \cite{KM}. We thank Niamh Farrell
for pointing this omission out to us.
Let $\bG$ be of type $E_8$ with a Frobenius endomorphism $F:\bG\rightarrow\bG$
such that $G=\bG^F=E_8(q)$. If $q\equiv\pm1\pmod6$ there
exists an isolated element $s\in G^*\cong G$ of order six with centraliser
$C_{\bG^*}(s)$ of type $A_5A_2A_1$. It was inadvertently left
out of \cite[Tab.~1]{KM} (probably as its order is divisible by two distinct
bad primes; but type $E_8$ has three bad primes). The centraliser of $s$ in
$\bG^*$ has rational type $A_5(q)A_2(q)A_1(q)$ if $q\equiv1\pmod6$,
and $\tw2A_5(q).\tw2A_2(q)A_1(q)$ if $q\equiv5\pmod6$. We parametrise the
$\ell$-blocks in $\cE_\ell(G,s)$ by $e$-cuspidal pairs for the only relevant
bad prime $\ell=5$.

\begin{thm}   \label{thm:E8,l=5}
 Theorems~{\rm 1.2, 1.4} and~{\rm1.5} of \cite{KM} continue to hold for all
 quasi-isolated $5$-blocks of $E_8(q)$ described above.
\end{thm}

\begin{proof}
The method is completely analogous to that employed in \cite{KM}. There are
three cases to distinguish, depending on whether $e:=e_5(q)$ is~1,2 or~4. First
we determine the decomposition of the Lusztig functor $\RLG$ for the relevant
$e$-cuspidal pairs $(\bL,\la)$ below $(\bG^F,s)$. Since the unipotent
characters of groups of type $A$ are uniform, and Jordan decomposition commutes
with Deligne--Lusztig induction, this decomposition follows from the known
corresponding one for unipotent characters. This also shows that $\bG^F$
satisfies an $e$-Harish-Chandra theory above each $e$-cuspidal pair below
$(\bG^F,s)$, thus showing the assertion of \cite[Thm.~1.4]{KM} in this case.
The decomposition is given in Tables~\ref{tab:q=1(5)}
and~\ref{tab:q=2(5)}. The notation is as in the analogous tables in \cite{KM}.
Here, the case $e=2$ can be obtained by Ennola duality from the one for $e=1$,
and the case of centraliser $\tw2A_5(q).\tw2A_2(q)A_1(q)$ when
$q\equiv\pm2\pmod5$
from the one with centraliser $A_5(q)A_2(q)A_1(q)$ when $q\equiv\mp2\pmod5$.
(Note that the $e$-Harish-Chandra series in Lines~1--5 are exactly as the
Lines~(1) and~(2) in \cite[Tab.~4]{KM}.)

\begin{table}[htbp]
\caption{Quasi-isolated 5-blocks in $E_8(q)$, $q\equiv1\pmod5$}   \label{tab:q=1(5)}
\[\begin{array}{|r|r|llll|}
\hline
 \text{No.}& C_{\bG^*}(s)^F& \bL^F& C_{\bL^*}(s)^F& \la& W_{\bG^F}(\bL,\la)\\
\hline\hline
 1& A_5(q)A_2(q)A_1(q)& \Phi_1^8& \bL^{*F}& 1& A_5\ti A_2\ti A_1\\
\hline
 2& \tw2A_5(q).\tw2A_2(q)A_1(q)& \Phi_1^5.A_1(q)^3& \Ph1^5\Ph2^3& 1& C_3\ti A_1\ti A_1 \\
\hline
 3&     & \Phi_1^4.D_4(q)& \Ph1^4\Ph2^2.\tw2A_2(q)& \phi_{21}& C_3\ti A_1\\
\hline
 4&     & \Phi_1^2.D_6(q)& \Ph1^2\Ph2.\tw2A_5(q)& \phi_{321}& A_1\ti A_1\\
\hline
 5&     & \Phi_1.E_7(q)& \Ph1.\tw2A_5(q).\tw2A_2(q)& \phi_{321}\otimes\phi_{21}& A_1\\
 \hline
\end{array}\]
\end{table}

\begin{table}[htbp]
\caption{Quasi-isolated 5-blocks in $E_8(q)$, $q\equiv\pm2\pmod5$}   \label{tab:q=2(5)}
\[\begin{array}{|r|r|llll|}
\hline
 \text{No.}& C_{\bG^*}(s)^F& \bL^F& C_{\bL^*}(s)^F& \la& W_{\bG^F}(\bL,\la)\\
\hline\hline
 6& A_5(q)A_2(q)A_1(q)& \Ph4.\tw2D_6(q)& \Ph1^3\Ph4.A_2(q)A_1(q)& 6\text{ chars}& Z_4\ti A_1\\
\hline
 7&     & \bG^F& C_{\bG^*}(s)^F& 18\text{ chars}& 1\\
\hline
\end{array}\]
\end{table}

It has been checked in \cite[Lemma~6.9]{KM} that all relevant $e$-split
Levi subgroups of $\bG$ satisfy $C_\bG(Z(\bL)_\ell^F)=\bL$.
(In fact they all already occur in Lines~19--23 of Table~7 respectively in
Line~43 of Table~8 in \cite{KM}.) Furthermore, all relevant $e$-cuspidal
characters $\la$ are readily seen to be of central $\ell$-defect. But then by
\cite[Prop.~2.13 and~2.15]{KM} the two conditions in \cite[Prop.~2.12]{KM}
are satisfied and so for all relevant $e$-cuspidal pairs $(\bL,\la)$ all
constituents of $\RLG(\la)$ lie in a single $5$-block $b_{\bG^F}(\bL,\la)$.
Moreover, $Z^\circ(\bL)^F \cap [\bL,\bL]^F$ is a $5'$-group, hence by
\cite[Prop.~2.7(g)]{KM}, in each case $(Z(\bL^F)_\ell,b_{\bL^F}(\la))$ is a 
centric $b_{\bG^F}(\bL,\la)$-Brauer pair. If $(\bL,\la)$ corresponds to
Line~1 of Table~\ref{tab:q=1(5)}, then by \cite[Prop.~2.7(c)]{KM}, a defect
group of
$b_{\bG^F}(\bL,\la)$ is an extension of $Z(\bL^F)_\ell$ by a Sylow $5$-subgroup
of $W_{\bG^F}(\bL,\la)$. In all other cases, the relative Weyl group is a
$5'$-group, hence by \cite[Prop.~2.7]{KM}, $(Z(\bL)^F_\ell,\la)$ is a maximal
$b_{\bG^F}(\bL,\la)$-Brauer pair, and in particular $Z(\bL)^F_\ell$ is a
defect group of $b_{\bG^F}(\bL, \la)$. Thus the defect groups of the various
blocks are as described in \cite[Thm~1.2]{KM}.
\par
Since the orders of the Sylow $5$-subgroups of the various $Z(\bL)^F$ in
Lines~2--5 are all distinct, these lines correspond to different blocks. To
see that the blocks represented by the six characters of Line~6 are distinct,
note that since $\bL =C_{\bG}(Z(\bL)^F_\ell)$ and since the pairs $(\bL,\la)$
are not $\bG^F$-conjugate, neither are the corresponding maximal Brauer pairs
$(Z(\bL)^F_\ell,\la)$. The blocks corresponding to Line~7 are all of defect
zero, hence are distinct. 
\end{proof}

\begin{rem}
Let us point out that in Table~\ref{tab:q=2(5)}, as in Table~8 of \cite{KM}
we suppressed the Ennola dual situations (obtained by changing $q$ to $-q$)
for which all Harish-Chandra series look completely similar since the
congruence of $q^2$ modulo~$5$ remains unchanged.
\end{rem}

The above results show that \cite[Thm.~A, Thm.~B]{KM15} remain unchanged
(note that Remark~2.2(4) of \cite{KM17} applies also to the isolated element
$s$ of order~6 above).
We also obtain the following consequences, completing the gap in the proofs
of Theorem~\cite[Thm.1.1]{KM} and \cite[Main Thm]{KM17} caused by the
missing case.

\begin{cor}
 For $s$ as above, the $5$-blocks in $\cE_5(\bG^F,s)$ with non-abelian defect
 group have characters of positive height. Further, $(\bG^F,\chi)$ is not a
 minimal counter-example to (HZC1) for any semisimple $5$-element
 $t\in \bG^{*F}$ commuting with $s$ and any $\chi\in\cE(\bG^F,st)$. 
\end{cor}

\begin{proof}
As already discussed in the proof of Theorem~\ref{thm:E8,l=5},
the block in Line~1 has non-abelian defect groups (as the relative Weyl group
has order divisible by~$5$) and the blocks in all other lines have abelian
defect groups. For the block in Line~1, the character in $\cE(\bG^F,s)$
corresponding to the unipotent character of $A_5(q)A_2(q)A_1(q)$ labelled by
$41\otimes2\otimes1$ has positive height. This proves the first assertion. 
\par
Suppose that $s$ corresponds to Lines~6--7. The blocks in Line~7 are all of
defect~0, and all remaining characters in $\cE_5(\bG^F,s)$ have the same
5-part in their degree, so are all of height~0 in their respective blocks.
In particular, the second assertion holds. Now suppose that $s$ corresponds 
to Lines~2--5. Here we may apply \cite[Lemma~8.5(3)(b)]{KM} in conjunction
with \cite[Prop.~8.6(1)]{KM} to conclude that the second assertion holds
(see the argument in the last part of the proof of \cite[Prop.~8.8]{KM}).
\end{proof}

\begin{rem}   \label{rem:3.14}
We take the opportunity to repeat what we already pointed out in the proof of
Theorem~3.14 of \cite{KM15}: in Table~4 of \cite{KM} each of the
lines~6,~7,~10,~11,~14 and 20 give rise to two $e$-cuspidal pairs and so to
two distinct $e$-Harish-Chandra series, but the two pairs also give rise to
different blocks, and similarly lines~2, 5, 8 and~11 in Table~3 of \cite{KM}
give rise to three $e$-Harish-Chandra series and three different blocks. We
thank Ruwen Hollenbach for bringing this misprint to our attention.
\end{rem}

%%%%%%%%%%%%%%%%%%%%%%%%%%%%%%%%%%%%%
\subsection{Further correction to the block distribution for $\ell=3$} \label{subsec:non-central def}

We discuss one more issue connected to the tables of block distributions
printed in \cite{KM}. In the accompanying statements, we say that the block
distribution is `indicated by the horizontal lines' in the tables, but in
fact, as we prove, the block distribution is related to the numbered lines,
in the sense that all unnumbered lines below a numbered line fall into the
block for the numbered line. This amended formulation fails in two places, though:
a misinterpretation of the statement of \cite{En00} on unipotent blocks of
groups of type $E_6$ led to a wrong assignment of certain $e$-Harish-Chandra
series to $3$-blocks. More concretely, we have the following:

\begin{prop}   \label{prop:corr non-central}
 In each of the following two cases and their Ennola duals, the
 $e$-Harish-Chandra series in the unnumbered line in the corresponding box of
 the table lies in the semisimple (first) block of the box:
 \begin{enumerate}[\rm(1)]
  \item $\bG^F=E_7(q)_\SC$, $\ell=3$, $C_{\bG^*}(s)^F=\Phi_1.E_6(q).2$, see
   \cite[Tab.~4]{KM}; and
  \item $\bG^F=E_8(q)$, $\ell=3$, $C_{\bG^*}(s)^F=E_7(q).A_1(q)$, see
   \cite[Tab.~6]{KM}.
 \end{enumerate}
 The corrected parts of the tables are thus as shown in Table~$\ref{tab:corr E7}$
 and~$\ref{tab:quasi-E8}$.
\end{prop}

\begin{table}[htbp]
\caption{Harish-Chandra series in some isolated 3-blocks of $E_7(q)_\SC$, $q\equiv1\pmod3$}   \label{tab:corr E7}
$$\begin{array}{|r|r|l|lll|}
\hline
 \text{No.}& C_{\bG^*}(s)^F& \bL^F& C_{\bL^*}(s)^F& \la& W_{\bG^F}(\bL,\la)\\
\hline
 8& \Ph1.E_6(q).2&     \Ph1^7& \bL^{*F}& 1& E_6.2\\
  &              &   \Ph1.E_6(q)& \bL^{*F}& E_6[\theta^{\pm1}]& 2\\
\hline
 9&              & \Ph1^3.D_4(q)& \bL^{*F}& D_4[1]& A_2.2\\
\hline
\end{array}$$
\end{table}

\begin{table}[htbp]
\caption{Harish-Chandra series in some isolated 3-blocks of $E_8(q)$, $q\equiv1\pmod3$}   \label{tab:quasi-E8}
$$\begin{array}{|r|r|l|lll|}
\hline
\text{No.}& C_{\bG^*}(s)^F& \bL^F& C_{\bL^*}(s)^F& \la& W_{\bG^F}(\bL,\la)\\
\hline
 3& E_7(q)A_1(q)&     \Ph1^8& \bL^{*F}& 1& E_7\ti A_1\\
  &             & \Ph1^2.E_6(q)& \bL^{*F}& E_6[\theta^{\pm1}]& A_1\ti A_1\\
\hline
 4&             & \Ph1^4.D_4(q)& \bL^{*F}& D_4[1]& C_3\ti A_1\\
\hline
 5&             &   \Ph1.E_7(q)& \bL^{*F}& E_7[\pm\xi]& A_1\\
\hline
\end{array}$$
\end{table}

\begin{proof}
The arguments given in the proofs of \cite{KM} apply verbatim, when using the
correct interpretation of the block distribution for $E_6(q)$ and $\tw2E_6(q)$
from \cite{En00}.
\end{proof}

The $3$-block $b$ labelled by Line 9 in Table~\ref{tab:corr E7} has non-abelian
defect groups,
and the argument for $b$ containing characters of different heights given in
\cite{KM17} is no longer valid in view of the correction above since we cannot
deduce the existence of a $3'$-character of positive height coming from a
$1$-Harish-Chandra series of non-central $3$-defect. We remedy this as
follows. Let $\bG\hookrightarrow\tilde\bG$ be a regular embedding. Then
$Z(\tbG^F)\bG^F$ is a normal subgroup of $\tbG^F$ of $3'$-index and
$Z(\tbG^F)\cap \bG^F$ is a $3'$-group. Hence it suffices to prove that a block
of $\tilde\bG^F$ covering $b$ has characters of different $3$-defect. By
Bonnaf\'e--Rouquier it then suffices to prove that the corresponding unipotent
$3$-block, say $B$, of a group $\bC^F$, with $\bC$ in duality with
$C_{\tbG^*}(\tilde s)$ ($\tilde s$ a preimage of $s$ in $\tbG^{*F}$) has
characters of different $3$-defect. Let $(\bL,\la)$ be a unipotent
$1$-cuspidal pair of $\bC^F$ defining $B$ and note that $\bC$ (a group of type
$E_6$) is a Levi subgroup of $\tbG$, hence has connected centre. Let
$t\in C_{\tbG^*}(\tilde s)$ be a $3$-element with centraliser of type $D_5$ and
$(\bL_t,\la_t)$ be a $1$-cuspidal pair of $C_{\tbG^*}(\tilde st)$ such that
$(\bL_t,\la_t)\to_t(\bL,\la)$ as in \cite[Thm~B]{En00}. Since
$C_{\tbG^*}(\tilde st)$ does not contain a Sylow 3-subgroup of
$C_{\tbG^*}(\tilde s)$, the characters in the $1$-Harish-Chandra-series of
$C_{\tbG^*}(\tilde st)^F$ defined by $(\bL_t,\la_t)$ have $3$-defect different
from that of the characters in $\cE(\bC^F,(\bL,\la))$. Then we are done by
\cite[Thm~B]{En00}.

%%%%%%%%%%%%%%%%%%%%%%%%%%%%%%%%%%%%%
\subsection{Proof of Theorems~\ref{thm:[KM13, 1.2]} and~\ref{thm:[KM15, Thm A]}}   \label{subsec:3.1 and 3.2}

\begin{proof}[Proof of Theorem~$\ref{thm:[KM13, 1.2]}$]
Suppose first that $s=1$. Then part~(c) follows from \cite{BMM} and parts~(a)
and~(b) follow from Theorems~A and~A.bis of \cite{En00} (see also
Lemma~\ref{lem:qcisc}). Parts~(a), (b) and~(c) in case $s$ is central follow
easily from the $s=1$ case. Now suppose that $s$ is non-central. If
$[\bG,\bG]$ is of type $G_2$, $F_4$ or $E_8$, then there is an $F$-stable
decomposition $\bG=[\bG,\bG]\times Z^\circ(\bG)$. Hence parts~(a),
(b) and~(c) follow from the analogous results in \cite{KM} and
Section~\ref{subsec:adjoint}, respectively.
\par
The first assertion of part~(f) follows as in Remark~2.2 and Section~4 of
\cite{KM15}. Note that if $\ell$ is good for $\bG$, then all unipotent
$e$-cuspidal pairs are of (quasi-)central $\ell$-defect. The second assertion
of part~(f) follows from this and by inspection of the tables of \cite{En00}
and \cite{KM} and of Tables~\ref{tab:quasi-E6}--\ref{tab:quasi-E8}.
\par 
Parts~(d) and~(e) are implicit in the construction of the tables cited above.
\end{proof}

\begin{proof}[Proof of Theorem~$\ref{thm:[KM15, Thm A]}$]
The proof of part (a) follows along the same lines as that of Theorem 3.4 of
\cite{KM15}, with some simplifications coming from the fact that $Z(\bX)$ is
connected. The only additional input required is the existence of an
$e$-Harish-Chandra theory at isolated elements and bad $\ell$ in the case
$[\bX,\bX]$ is of exceptional type which is provided in this section. Here note
that Lemma~3.1 of \cite{KM15} is stated for all connected reductive groups.
Part~(b) is immediate from part~(a). Part~(c) follows from~(a) and the proof of
Theorem~3.6 of \cite{KM15}. Part~(d) follows from part~(b) and
\cite[Lemmas~2.3 and~3.7, Thm~A(c)]{KM15} and the remarks following
Definition~2.12 of \cite{KM15}.
\end{proof}

%%%%%%%%%%%%%%%%%%%%%%%%%%%%%%%%%%%%%
\subsection{Decomposition of $\RLG$} \label{subsec:RLG}
We take the opportunity to resolve the last ambiguities left in the
determination of the decomposition of $\RLG(\la)$ for certain unipotent
characters $\la$ in exceptional groups of Lie type, viz.~the cases denoted
``15+16'', ``40+41" and ``42+43'' in \cite[Tab.~2]{BMM} (note that since the
statements only concern unipotent characters, the precise isogeny type of
$\bG$ is not relevant):

\begin{lem}   \label{lem:dec RLG}
 \begin{enumerate}[\rm(a)]
  \item Let $\bL$ be a Levi subgroup of rational type $\tw2E_6(q).\Phi_2$ in
   $\bG$ of type $E_7$. Then
   $$\RLG(\tw2E_6[\theta^j])=(E_6[\theta^j],1) - (E_6[\theta^j],\eps)
     \quad\text{for $j=1,2$}.$$
  \item Let $\bL$ be a Levi subgroup of rational type $\tw2E_6(q).\Phi_2^2$ in
   $\bG$ of type $E_8$. Then
   $$\RLG(\tw2E_6[\theta])=(E_6[\theta],\phi_{1,0})-(E_6[\theta],\phi_{1,3}')-(E_6[\theta],\phi_{1,3}'')+(E_6[\theta],\phi_{1,6})-2E_8[-\theta]-2E_8[\theta].$$
  \item Let $\bL$ be a Levi subgroup of rational type $E_7(q).\Phi_2$ in
   $\bG$ of type $E_8$. Then
   $$\RLG(\phi_{512,11})=\phi_{4096,11}-\phi_{4096,26}\quad\text{and}\quad
     \RLG(\phi_{512,12})=\phi_{4096,12}-\phi_{4096,27}.$$
 \end{enumerate}
\end{lem}

\begin{proof}
In all three cases this follows from Shoji's explicit description of Lusztig
induction of unipotent characters in terms of the Fourier transform
\cite{Sh87}. In cases~(a) and~(c) we can also provide a block theoretic
argument. It will be sufficient to show the claim for one value of $q$ since by
\cite[Thm~1.33]{BMM} the decomposition of Lusztig induction of unipotent
characters is generic. We choose $q$ such that $q+1$ is divisible by a prime
$\ell$ larger than~7, say $q=37$, and take $\ell=19$, so $e_\ell(q)=2$. The
four unipotent characters $\la=
\tw2E_6[\theta],\tw2E_6[\theta^2],\phi_{512,11},\phi_{512,12}$ considered are
2-cuspidal in their respective Levi subgroups and thus of $\ell$-defect zero.
It then follows by \cite[Thm]{CE94} that all constituents of $\RLG(\la)$ lie
in the same $\ell$-block of $\bG^F$. On the other hand, the two characters $\la$
given in either case are not conjugate under any automorphism of $\bL^F$
\cite[Thm~4.5.11]{GM20} and thus by \cite{CE94} they label distinct blocks of
$\bG^F$. It was already shown in \cite{BMM} that
in all cases $\RLG(\la)$ has norm~2, and that the constituents are among the
ones listed in the statement. Thus we are done if we can show that the stated
decompositions agree with the $\ell$-block distribution. \par
For~(a) consider the unipotent character $\rho:=E_6[\theta]$ of the split Levi
subgroup $\bL_1$ of $\bG=E_7$ of rational type $E_6(q).\Phi_1$. It is
1-cuspidal, and its
relative Hecke algebra in $\bG^F$ is of type $A_1$ with parameter~$q^9$
\cite[Tab.~4.8]{GM20}. Thus its Harish-Chandra induction $R_{\bL_1}^\bG(\rho)$
is indecomposable modulo all primes $\ell$ dividing~$q^9+1$ by
\cite[Prop.~3.1.29]{GM20}, whence its two constituents $E_6[\theta],1$ and
$E_6[\theta],\epsilon$ lie in the same $\ell$-block of $\bG^F$.
In case~(c) it follows from \cite[Table~F.6]{GP} that the block distribution
is as claimed.
\end{proof}

%%%%%%%%%%%%%%%%%%%%%%%%%%%%%%%%%%%%%%%%%%%%%%%%%%%%%%%%%%%%%%%%%%%%%%%%%
\section{Robinson's conjecture on defects}   \label{sec:Rob}
Robinson's conjecture on character defects \cite{Ro96} presented in the
introduction was reduced to the case of minimal counter-examples in
quasi-simple groups in \cite[Thm~2.3]{FLLMZ} based on work of Murai, and it was
shown to hold for all odd primes $\ell$ and for the $2$-blocks of any
quasi-simple group not of exceptional Lie type in odd characteristic in
\cite{FLLMZ,FLLMZb}.
Here, we investigate its validity in the following situation. Let $\bG$ be a
simple algebraic group of simply connected type with a Frobenius map $F$ such
that $G=\bG^F$ is of type $G_2(q)$, $\tw3D_4(q)$, $F_4(q)$, $E_6(\pm q)$,
$E_7(q)$ or $E_8(q)$, for some odd prime power $q$. Let $S$ be a central
quotient of~$G$. Note that $G=S$ unless possibly when $G$ is of type $E_6$ or
$E_7$. Let $\bar B$ be a block of $S$ dominated by an isolated $2$-block $B$
of~$G$.

\begin{lem}   \label{lem:RobRed}
 Suppose that $\bG$ as above is of type $E_6$ and let $\bG\hookrightarrow\tbG$
 be a regular embedding. Let $S$, $B$ and $\bar B$ be as above and let
 $\tilde B$ be a block of $\tbG^F$ covering $B$. Then any one of $B$, $\bar B$
 and $\tilde B$ satisfies Robinson's conjecture if and only if any of the other
 two does.
\end{lem}

\begin{proof} 
Since $Z(G)$ is a $2'$-group, $B$ and $\bar B$ are isomorphic blocks, and in
particular $B$ and $\bar B$ have isomorphic defect groups and the same set of
character degrees. Thus $B$ satisfies Robinson's conjecture if and only if
$\bar B$ does. Again, since $Z(G)$ is a $2'$-group and since
$[\tbG,\tbG] = \bG$, $Z(\tbG)^F\bG^F$ is a normal subgroup of odd index in
$\tbG^F$, and thus $\tilde B$ satisfies Robinson's conjecture if and only if
some (hence any) block, say $C$, of $Z(\tbG)^F\bG^F$ covering $B$ satisfies the
conjecture. Since $Z(\tbG)^F$ is central in $Z(\tbG)^F\bG^F$, $B$ and $C$ have
the same set of character degrees. Moreover, since
$Z(\bG^F)=Z(\tbG)^F\cap\bG^F$ is a $2'$-group, if $D$ is a defect group of $B$,
then $Z(\tbG)^F_2 D \cong Z(\tbG)^F_2 \times D $ is a defect group of $C$. So
$C$ satisfies Robinson's conjecture if and only if $B$ does.
\end{proof} 

\begin{prop}   \label{prop:Robinson unip}
 Let $B$ be a unipotent $2$-block of $G$ as above. Then the block $\bar B$ of
 $S$ is not a minimal counter-examples to Robinson's conjecture.
\end{prop}

\begin{proof}
First assume $B$ is the principal block of $G$. Then its defect groups are the
Sylow 2-subgroups of $G$. Their centres are given in Table~\ref{tab:exc syl};
here $T_\eps$ denotes a torus of order $q-\eps1$, and
$\overline{E_7(q)}:=E_7(q)/Z(E_7(q))$.

\begin{table}[htb]
\caption{Centres of Sylow $2$-subgroups $P\in\Syl_2(S)$}   \label{tab:exc syl}
$$\begin{array}{|c|cc||c|cc|}
 S& C_G(t)& Z(P)& S& C_G(t)& Z(P)\\
 \noalign{\hrule}
     G_2(q)&        A_1(q)^2&    C_2&  \tw2E_6(q)& \tw2D_5(q).T_-& C_{|q+1|_2}\\
 \tw3D_4(q)&  A_1(q^3)A_1(q)&    C_2&  E_7(q)& D_6(q)A_1(q)& C_2^2\\
     F_4(q)&          B_4(q)&    C_2&  \overline{E_7(q)}& D_6(q)A_1(q)& C_2\\
     E_6(q)& D_5(q).T_+& C_{|q-1|_2}&  E_8(q)& D_8(q)& C_2\\
\end{array}$$
\end{table}

This is obtained as follows. If $P\in\Syl_2(G)$ and $t\in Z(P)$ then $C_G(t)$
has odd index in $G$ and $Z(P)\le P\le C_G(t)$. The centralisers of semisimple
elements in $G$ can be enumerated with the algorithm of Borel--de Siebenthal
(see e.g. \cite[13.2]{MT}).
It turns out that the only centralisers $C_G(t)$ of 2-elements
$t\in G\setminus Z(G)$ of odd index in $G$ are as listed in
Table~\ref{tab:exc syl}. (In fact, these can also be found on the website
\cite{Lue}.) Then $|Z(P)|$ can be read off from the structure of $C_G(t)$.
If $S=\overline{E_7(q)}$ then we may consider $S$ as the derived subgroup of an
adjoint type group, and with the same argument as before we find that
$|Z(P)|=2$ in this case.

By \cite[Lemma~3.1]{FLLMZ} Robinson's conjecture holds when $|Z(P)|=2$. Thus,
by Table~\ref{tab:exc syl} we only need to concern ourselves with $S$ of type
$E_6(q)$, $\tw2E_6(q)$, or $S=E_7(q)$. Let's first assume that $\bG$ is of
type $E_6$. Then by Lemma~\ref{lem:RobRed} we may instead argue for the
principal block $\tilde B$ of $\tbG^F$. Now, by \cite[Thm~12]{CE93} for every
$\chi\in\Irr(\tilde B)$ there is a character of the same height in
the principal block of
$G_\ad:=\tilde G/Z(\tilde G)$, a group of adjoint type. Note that $G$ and
$\tilde G/Z(\tilde G)$ have isomorphic Sylow 2-subgroups as $G$ is a central
extension of degree dividing~3 of the derived subgroup of $G_\ad$. We may hence
argue for the principal 2-block of $G_\ad$; here $G_\ad^*\cong G$.
\par
Now, according to Enguehard's description
in \cite[Thm~B]{En00} a character $\chi\in\cE(G_\ad,t)$ lies in the principal
2-block if and only if $t\in G\cong G_\ad^*$ is a 2-element and moreover the
Jordan correspondent of~$\chi$ in $\cE(C_G(t),1)$ lies in a Harish-Chandra
series with Harish-Chandra vertex either a torus or a Levi subgroup of
type~$D_4$. First assume that $\chi$ has Harish-Chandra vertex $D_4$. Then
$\bH:=C_\bG(t)$ has a Levi subgroup of type $D_4$ and hence is either
of type $D_4$, $D_5$ or $E_6$. In either case, \cite[Prop.~6.2]{FLLMZ}
shows that $(|S|/\chi(1))_2\ge(q-1)_2^2>|Z(P)|$ and we are done. If $\chi$ has
trivial Harish-Chandra vertex then by the same argument we are done when $\bH$
has $\FF_q$-rank at least~2. Note that $\bH$ has $\FF_q$-rank at least~1 as by
Table~\ref{tab:exc syl} every 2-element of $\bG_\ad^*$ centralises a split
torus of rank~1. Now if $\bH$ has $\FF_q$-rank~1, then all of its unipotent
characters have odd
degree, and it is easy to see that $(|S|/\chi(1))_2\ge2(q-1)_2>|Z(P)|$.
\par
The same discussion applies to $G=\tw2E_6(q)$ by interchanging the cases
corresponding to the two possible congruences of $q$ modulo~4.
\par
Now assume $S=G=E_7(q)$. Let $\chi$ lie in the principal 2-block $B$ of $G$,
so in $\cE(G,t)$ for some 2-element $t\in G^*$. If $t=1$, so $\chi$ is
unipotent, then by inspection $\df(\chi)\ge3$ unless $\chi$ lies in an ordinary
Harish-Chandra series above a Levi subgroup of type $E_6$, but by
\cite[p.~354]{En00} these are not in the principal 2-block of $G$. Now
assume $t\ne1$, and let $t_1$ be the involution in $\langle t\rangle$. Then
$C_{G^*}(t_1)$ has one of the structures $D_6(q)A_1(q)$, $A_7(\pm q).2$ or
$E_6(\pm q)(q\mp1).2$. Assume $C_{G^*}(t)$ involves an $E_6$-factor. Let
$\bG\hookrightarrow\tbG$ be a regular embedding, $\tilde t\in\tbG^{*F}$ a
preimage of~$t$ of 2-power order, and let $\tilde B$ be the principal block
of~$\tbG^F$. By \cite[Thm~B]{En00} the characters in $\tilde B$ in series
$\tilde t$ are those in Jordan correspondence with unipotent characters of
$C_{\tbG^{*F}}(\tilde t)$ with a rational Frobenius eigenvalue. Thus, the
characters in $B$ in series $t$ are again among those in Jordan correspondence
with unipotent characters with a rational Frobenius eigenvalue, and by
inspection all of these 
have defect at least~3. So we may assume that $C_{G^*}(t)$ has only factors of
classical type. Again by using the lists of unipotent character degrees and
arguing as before, one sees that all characters in $\cE(G,t)\cap\Irr(B)$
%Theorem~\ref{thm:thmA}, 
are of defect at least~3, so $B=\bar B$ is not a counter-example.
\par
The non-principal unipotent 2-blocks of groups of exceptional type were
determined in \cite{En00}. In Table~\ref{tab:exc np} we list those blocks
having non-abelian defect groups, as well as properties of their defect groups
$D$ which have also been taken from \cite{En00}. We label the blocks by their
Harish-Chandra vertex $(\bL,\la)$ of quasi-central defect, in the notation
of loc.~cit.

\begin{table}[htb]
\caption{Non-principal unipotent $2$-blocks of non-abelian defect}   \label{tab:exc np}
$$\begin{array}{c|c|ccc}
 \noalign{\hrule}
 G& \text{cond.}& ([\bL,\bL],\la)& D& |Z(D)|\\
\noalign{\hrule}
 E_7(q)& q\equiv1\,(4)& (E_6,E_6[\theta]),(E_6,E_6[\theta^2])& Z_{(q-1)_2}.2& 2\\
       & q\equiv3\,(4)& (\tw2E_6,\tw2E_6[\theta]),(\tw2E_6,\tw2E_6[\theta^2])& Z_{(q+1)_2}.2& 2\\
 E_8(q)& q\equiv1\,(4)& (E_6,E_6[\theta]),(E_6,E_6[\theta^2])& Z_{(q-1)_2}^2.2^2& \le4\\
       & q\equiv3\,(4)& (\tw2E_6,\tw2E_6[\theta]),(\tw2E_6,\tw2E_6[\theta^2])& Z_{(q+1)_2}^2.2^2& \le4\\
\noalign{\hrule}
\end{array}$$
\end{table}

Since the conjecture holds for blocks with $|Z(D)|=2$ by
\cite[Lemma~3.1]{FLLMZ}, we only need to consider the 2-blocks $B$ in $E_8(q)$.
First from the list of unipotent degrees it is easy to check that all unipotent
characters $\chi$ in these blocks, as described in \cite[Thm~A]{En00}, have
$\df(\chi)\ge3$.
By \cite[Thm~B]{En00} all characters in $\cE_2(G,1)\cap\Irr(B)$ have
Harish-Chandra vertex of type $E_6$. Thus the centralisers of the relevant
2-elements $1\ne t\in \bG^{*F}$ contain either an $E_6$- or an
$E_7$-factor. Again, it is straightforward to verify that all such characters
$\chi$ have $\df(\chi)\ge3$, whence our claim holds.
\end{proof}

We suspect that in fact the defect groups of the blocks $B$ for $E_8(q)$ as in
Table~\ref{tab:exc np} are isomorphic to Sylow 2-subgroups of $G_2(q)$, in
which case their centres would have order~2 and the proof would be even easier.

\begin{prop}   \label{prop:Robinson isolated1}
 Let $B$ be an isolated, non-unipotent $2$-block of $G$ in $\cE_2(G,s)$, with
 $\bG$ not of type $E_7$. Assume $\bC^*:=C_{\bG^*}(s)$ has only factors of
 classical type. Then the block $\bar B$ of $S$ is not a minimal
 counter-example to Robinson's conjecture.
\end{prop}

\begin{proof}
Let $B$ be as in the statement. First assume $\bG$ is not of type $E_6$. Then
$S=G$, $\bC^*$ is connected, and Jordan decomposition on the level of $G$ and
of $C=\bC^F$ together yield a defect preserving bijection
$\cE_2(G,s)\to\cE_2(C,1)$. Since $\bC^*$ (and hence also $\bC$) has only
factors of classical type, $\cE_2(C,1)$ is a single 2-block, namely the
principal block $b_0$ of $C$, and by Proposition~\ref{prop:one block} so is
$\cE_2(G,s)$. In particular, $B$ is a \emph{minimal block}, that is, $B$
contains a semisimple character in $\cE(G,s)$. Then by \cite[Prop.~E]{Ruh22} a
defect group of $B$ is isomorphic to a Sylow 2-subgroup of $C$. Thus, if $B$ is
a counter-example then so is $b_0$. Since $C$ is strictly smaller than $G$, this
shows that $B$ cannot be minimal.   \par
Finally, if $\bG$ is of type $E_6$, then by Lemma~\ref{lem:RobRed} we may
replace $B$ by a block $\tilde B$ of $\tilde G$ instead and the same arguments
go through. 
\end{proof}

\begin{lem} \label{lem:Robinson E6E2}
 Suppose that $\bG$ is of type $E_8$. Let $s\in\bG^*$ be isolated such that
 $C_{\bG^*}(s)$ is of type $E_6A_2$ and let $\bC$ be dual to
 $\bC^*:= C_{\bG^*}(s)$. There is a $2$-defect preserving bijection
 $\psi:\cE_2(\bG^F,s)\to\cE_2(\bC^F,1)$ which preserves $2$-blocks.
\end{lem} 

\begin{proof}
Since the centre of a simply-connected covering of $\bC$ is a $3$-group, by
\cite[Thm 12]{CE93}, we may replace $\bC$ by
its adjoint quotient $\bar\bC = \bC/Z(\bC)$. By inspection of \cite{KM} and
\cite{En00}, both $\cE_2(\bG^F,s)$ and $\cE_2(\bar\bC^F,1)$ contain exactly two
2-blocks. Let $B$ be the minimal $2$-block in $\cE_2(\bG^F,s)$. Now by
Theorem~\ref{thm:thmA}, for any $2$-element $t\in C_{\bG^*}(s)^F$, the
characters of $\cE(\bG^F, st)$ which lie in $B$ are precisely those whose
Jordan correspondents in $\cE(C_{\bG^*}(st)^F,1)$ lie in $e$-Harish-Chandra
series above $e$-cuspidal characters with rational Frobenius eigenvalue. The
same description applies to the characters in the principal block
of~$\bar\bC^F$
by \cite[Thm~B]{En00}. The result follows by noting that the natural surjection
$\bar \bC^*\to\bC^*$ induces a bijection from the set of conjugacy classes of
$2$-elements of $\bar\bC^{*F}$ to those of $\bC^{*F}$ and that this surjection
induces isogenies between centralisers of corresponding elements, so in
particular preserves their orders.
\end{proof} 

\begin{prop}   \label{prop:Robinson isolated2}
 Let $B$ be an isolated, non-unipotent $2$-block of $G$. Then the block
 $\bar B$ of $S$ is not a minimal counter-example to Robinson's conjecture.
\end{prop}

\begin{proof}
Let $B$ be isolated but not unipotent, labelled by a semisimple $2'$-element
$1\ne s\in G^*$. By inspection of the tables in \cite{KM} the only cases
where $\bC^*=C_{\bG^*}(s)$ has factors of exceptional type are for $\bG$ of
type $E_8$ and $\bC^*$ of type $E_6A_2$. Thus by Proposition~\ref{prop:Robinson
isolated1} we only need to discuss these blocks, and the case when $\bG$ has
type $E_7$.
\par
First assume $\bG$ is of type $E_7$. According to \cite[Tab.~4]{KM}, the
relevant 2-blocks are those with $\bC^*:=C_{\bG^*}(s)$ of rational type
$A_5(q)A_2(q)$ or $\tw2A_5(q).\tw2A_2(q)$ (depending on the congruence of $q$
modulo~$3$). In both cases, $\cE_2(G,s)$ is a single 2-block by \cite{KM}, with
defect group $D$ isomorphic to a Sylow 2-subgroup of $C_{\bG^{*F}}(s)^*$ by
\cite[Prop.~E]{Ruh22}. By Ennola duality, it suffices to consider the case
$C^*:=\bC^{*F}\cong A_5(q)A_2(q)$. Assume $q\equiv1\pmod4$. Then,
$|Z(D)|=2(q-1)_2^2$.
By the description of centralisers of elements $t$ of 2-power order in linear
groups, any character in Lusztig series $\cE(C^*,t)$ has defect at least
$\log_2(8(q-1)_2^2)$. Since the characters in $\cE_2(G,s)$ are in height
preserving bijection to the characters of $\cE_2(C,1)$ via Jordan
decomposition, where $C:= \bC^F$, $B$ is not a counter-example. Now consider
the corresponding block $\bar B$ of $S=G/Z(G)$, with defect group
$\bar D=D/Z(G)$. By inspection, $|Z(\bar D)|\le(q-1)_2^2$. Since any
character of $\bar B$ is a character of $B$, and defects decrease by~1,
neither is the block $\bar B$ a counter-example.
Now assume $q\equiv3\pmod4$. Then $|Z(D)|=2^3$. Here any character in
$\cE(C^*,t)$ has defect at least $5$. Furthermore, $\bar D=D/Z(G)$ has
centre of order~$2^2$. Thus we may conclude as in the previous case.
\par
It remains to consider the isolated blocks of $E_8$ in series $s$ with
$\bC^*$ of type $E_6A_2$. The defect groups of the minimal block
$B$ in $\cE_2(G,s)$ are isomorphic to Sylow 2-subgroups of $C$, by
\cite[Prop.~E]{Ruh22}. Hence by Lemma~\ref{lem:Robinson E6E2} and its proof,
if $B$ is a counter-example, then so is the principal block of $\bC^F$,
contradicting the minimality of $B$.
\par
Finally, let $B$ be the minimal block in $\cE_2(G,s)$ as above, see
Table~\ref{tab:Rob-E8} (taken from \cite[Tab.~5]{KM}) and their Ennola duals.
According to \cite[Thm~1.2]{KM}, the defect groups for $B$ the block No.~6 are
meta-cyclic and hence Robinson's conjecture holds by \cite[Cor.~8.2]{Sam14}.

\begin{table}[htbp]
\caption{Isolated non-unipotent non-maximal 2-blocks in $E_8(q)$, $q\equiv1\pmod4$}   \label{tab:Rob-E8}
\[\begin{array}{|c|r|llll|}
\hline
 \text{No.}& C_{\bG^*}(s)^F& \bL^F& C_{\bL^*}(s)^F& \la& W_{\bG^F}(\bL,\la)\\
\hline\hline
 4&         E_6(q).A_2(q)& E_6& \bL^{*F}& E_6[\theta^{\pm1}]& A_2\\
\hline
 6& \tw2E_6(q).\tw2A_2(q)& E_7& \Ph1\Ph2.\tw2E_6(q)& \tw2E_6[\theta^{\pm1}]& A_1\\
  &              & E_8& C_{\bG^*}(s)^F& \tw2E_6[\theta^{\pm1}]\otimes\phi_{21}& 1\\
\hline
\end{array}\]
\end{table}

Now let $B$ be the No.~4 block. As before, let $\bC\le\bG$ be $F$-stable of
type $E_6A_2$, corresponding to $\bC^*$ under an identification of $\bG$ with
$\bG^*$. Let $b_{\bL^F}(\la)$ be the $2$-block of $\bL^F$
containing~$\la$. As explained in the proof of \cite[Prop.~6.4]{KM},
$(Z(\bL)^F_2,b_{\bL^F}(\la))$ is a $B$-Brauer pair. Moreover, by
\cite[Lemma~6.2]{KM}, $C_{\bG^F}(Z(\bL)^F_2) = \bL^F\leq \bC^F$. So, by
general block-theoretic considerations, letting $d$ be the unique $2$-block of
$\bC^F$ such that $(Z(\bL)^F_2, b_{\bL^F}(\la))$ is a $d$-Brauer pair,
there is a defect group $R$ of $d$ and a defect group $P$ of $B$ such
that $Z(\bL)^F_2 \leq R \leq P$. Also, note from the table entry for $B$ that
$Z(\bL)^F_2 $ is of index $2$ in~$P$. Thus, either $R=Z(\bL)^F_2$ or $R=P$.
Since $W_{\bG^F}(\bL,\la)\cong W_{\bC^F}(\bL,\la)$ is not a $2'$-group, $R$ is
not abelian by \cite[Prop.~2.7(e)]{KM} and hence $R=P$.
\par
Let $\bC_1$ be the $E_6$-subgroup of $\bC$ and $\bC_2$ the $A_2$-subgroup of
$\bC$. Then $\bC_1^F\bC_2^F$ is normal of index $3$ in $\bC^F$, $\bC_1^F$ and
$\bC_2^F$ commute and $\bC_1 ^F\cap \bC_2^F$ is of order~$3$. It follows that
$R = R_1 R_2\cong R_1\times R_2$ with $R_i$ a defect group of a $2$-block of
$\bC_i^F$, $i=1,2$. Since $Z(\bL)^F_2\leq R \cap \bC_2^F =R_2$ is of index~$2$
in $R$ and $Z(\bL)^F_2$ is self-centralising in $R$, it follows that $R_1$ is
trivial and hence $Z(\bL)^F_2$ is of index $2$ in $P=R=R_2$. Since
$|\bC^F_2|_2= 2|Z(\bL)^F_2|$, $P=R_2$ is a Sylow $2$-subgroup of $\bC_2^F$.
Now the result follows by Lemma~\ref{lem:Robinson E6E2} and its proof.
\end{proof}

Complementing earlier investigations in \cite{FLLMZ,FLLMZb} this proves
Robinson's conjecture:

\begin{proof}[Proof of Theorem~$\ref{thm:Robinson}$]
According to \cite{FLLMZ,FLLMZb} a minimal counter-example would have to
occur as a quasi-isolated 2-block of an exceptional type quasi-simple group
$S$ in odd characteristic. Note that the exceptional covering group $3.G_2(3)$
was handled in \cite[Thm.~3.6]{FLLMZ}. By the defect group preserving Morita
equivalences from \cite{BDR17}, we can further restrict to isolated blocks.
Note that here the centre of $\bG^F$ is always cyclic, so the known gap in
\cite{BDR17} is not relevant.
By the main result of \cite{KM} Robinson's conjecture holds for blocks with
abelian defect groups, so we need not consider the Ree groups $^2G_2(q^2)$.
But now Propositions~\ref{prop:Robinson unip} and~\ref{prop:Robinson isolated2}
give the claim.
\end{proof}

%%%%%%%%%%%%%%%%%%%%%%%%%%%%%%%%%%%%%%%%%%%%%%%%%%%%%%%%%%%%%%%%%%%%%%%%%

\end{document}